\documentclass[a4paper, 12pt, reqno]{amsart}

\usepackage{amsmath}
\usepackage{amssymb}
\usepackage{amsthm}
\usepackage{color}
\usepackage{graphics}
\usepackage{colortbl}

\usepackage{amscd}
\usepackage{bm}
\usepackage{Here}

\title{$I$-functions of Calabi--Yau 3-folds in Grassmannians}
\date{}

\author[D.~Inoue]{Daisuke Inoue}
\address{
Graduate School of Mathematical Sciences,
The University of Tokyo,
3-8-1 Komaba,
Meguro-ku,
Tokyo,
153-8914,
Japan.}
\email{ino@ms.u-tokyo.ac.jp}

\author[A.~Ito]{Atsushi Ito}
\address{
Department of Mathematics,
Graduate School of Science,
Kyoto University,
Kyoto 606-8502,
Japan.
}
\email{aito@math.kyoto-u.ac.jp}

\author[M.~Miura]{Makoto Miura}
\address{
Korea Institute for Advanced Study,
85 Hoegiro,
Dongdaemun-gu,
Seoul,
130-722,
Republic of Korea.
}
\email{miura@kias.re.kr}

\usepackage{latexsym}

\def\linseq{16}
\def\dnc{18}
\def\sqmix{25}

\newtheorem{Theorem}{Theorem}[section]

\newtheorem{Proposition}[Theorem]{Proposition}
\newtheorem{Lemma}[Theorem]{Lemma}

\newtheorem{Fact}[Theorem]{Fact}

\newtheorem*{Main results}{Main results}

\theoremstyle{definition}
\newtheorem{Assumption}[Theorem]{Assumption}
\newtheorem{Definition}[Theorem]{Definition}
\newtheorem{Example}[Theorem]{Example}

\theoremstyle{remark}
\newtheorem{Remark}[Theorem]{Remark}

\DeclareMathOperator{\Pic}{Pic}
\DeclareMathOperator{\Sym}{Sym}
\DeclareMathOperator{\NE}{NE}
\DeclareMathOperator{\Hom}{Hom}
\DeclareMathOperator{\ev}{ev}
\DeclareMathOperator{\ann}{ann}
\DeclareMathOperator{\Mat}{Mat}

\theoremstyle{definition}
\newtheorem*{acknowledgements}{Acknowledgements}

\begin{document}
\begin{abstract}
We study $I$-functions of Calabi--Yau 3-folds with Picard number one which are zero loci of general sections of direct sums of 
globally generated irreducible homogeneous vector bundles on Grassmannians. 
\end{abstract}

\maketitle

\section{Introduction}
A smooth projective variety $Y$ with trivial canonical bundle and $H^i(Y, \mathcal{O}_{Y}) = 0$ $(1 \leq i \leq \dim Y - 1)$ is called a Calabi--Yau manifold.
Constructing Calabi--Yau manifolds is of considerable interest to both algebraic geometers and theoretical physicists, since many different topology types are possible in higher dimensions ($\dim Y \geq 3$) and also we observe a mysterious symmetry (mirror symmetry) among them.
Recently the present authors have obtained a complete list of Calabi--Yau 3-folds which are given by the zero loci of general sections of
direct sums of globally generated irreducible homogeneous vector bundles over Grassmannians \cite{IIM}.
In this paper,
we focus on such Calabi--Yau 3-folds of Picard number one,
which are listed in Table \ref{cylist},
and determine differential operators which characterize the $I$-functions of the Calabi--Yau 3-folds.

In \cite{Kuc}, K\"{u}chle obtained a complete list of Fano 4-folds with Fano index one which are given by the zeros of general sections of direct sums of globally generated irreducible homogeneous vector bundles over Grassmannians, which we call K\"{u}chle's Fano 4-folds. Many of our Calabi--Yau 3-folds in Table \ref{cylist} are naturally given by the anticanonical section of K\"{u}chle's Fano 4-folds. It should be noted, however, that some of our Calabi--Yau 3-folds do not factor through K\"{u}chle's Fano 4-folds.

\vspace{2mm}
In \cite{CCGK}, they determine the quantum periods of all three-dimensional Fano manifolds. 
In this paper, we follow their method to calculate $I$-functions of Calabi--Yau $3$-folds.
An $I$-function is a certain function which determines the $J$-function of quantum cohomology of a Calabi--Yau 3-fold.
When we restrict our attention to Calabi--Yau 3-folds of Picard number one, the $I$-functions satisfy certain differential equations of fourth order (Picard--Fuchs equations) and play a central role in the proof of the mirror theorem \cite{Giv1}, \cite{LLY}.
We can use this relation to determine the $I$-functions.
For the other Calabi--Yau 3-folds in Table \ref{cylist}, not factoring through K\"{u}chle's Fano 4-folds, we use the so-called abelian/nonabelian correspondence in \cite{CFKS} which is applicable to the (twisted) $I$-functions of homogeneous vector bundles on Grassmannians.
For most of them, straightforward applications of abelian/nonabelian correspondence are possible to determine the $I$-functions.

For all Calabi--Yau 3-folds in Table \ref{cylist}, except for No.\ \dnc, either of the two methods above can be used to determine the $I$-functions. 
To handle the remaining case No.\ \dnc,
we use the fact that Calabi--Yau 3-folds of type No.\ \dnc \ are isomorphic to zero loci of sections of suitable vector bundles on a geometric invariant theory quotient,
which is known as the determinantal nets of conics \cite{EPS}.
Although it is not known whether the abelian/nonabelian correspondence holds for the determinantal nets of conics or not,
we compute a conjectural $I$-function for No.\ \dnc \ under the assumption of the validity of the abelian/nonabelian correspondence for determinantal nets of conics.

\begin{table}[]\label{cylist}
\begin{tabular}[c]{|l|l|l|r|r|r|l|l|}
\hline
No. & Grassmann & vector bundle & $H^3$ & $c_2.H$ & $c_3$ & K\"uchle &  
Database \\ \hline \hline
1 & $G(2,4)$ & $\mathcal{O}(4)$ & $8$ & $56$ & $-176$ & & $6$ \\
2 & $G(2,5)$ & $\mathcal{O}(1) \oplus \mathcal{O}(2)^{\oplus 2}$ & $20$ & $68$ & $-120$ & (b2) & $25$ \\
3 & $$ & $\mathcal{O}(1)^{\oplus 2} \oplus \mathcal{O}(3)$ & $15$ & $66$ & $-150$ & (b1) & $24$ \\
4 & $$ & $\mathcal{S}^{\ast}(1) \oplus \mathcal{O}(2)$ & $24$ & $72$ & $-116$ & & $29$ \\
5 & $$ & $\Lambda^2 \mathcal{Q}(1)$ & $25$ & $70$ & $-100$ &  & $101$ \\
6 & $G(2,6)$ & $\mathcal{O}(1)^{\oplus 4} \oplus \mathcal{O}(2)$ & $28$ & $76$ & $-116$ & (b6) & $26$ \\
7 & $$ & $\mathcal{S}^{\ast}(1) \oplus \mathcal{O}(1)^{\oplus 3}$ & $33$ & $78$ & $-102$ & (b5), I & $198$ \\
10 & $$ & $\mathcal{Q}(1) \oplus \mathcal{O}(1)$ & $42$ & $84$ & $-98$ & (b3), II & $27$ \\
12 & $G(2,7)$ & $\mathcal{O}(1)^{\oplus 7}$ & $42$ & $84$ & $-98$ & (b7), III & $27$ \\
13 & $$ & $\Sym^2 \mathcal{S}^{\ast} \oplus \mathcal{O}(1)^{\oplus 4}$ & $56$ & $92$ & $-92$ & (b8), V & $212$ \\
15 & $$ & $\Lambda^4 \mathcal{Q} \oplus \mathcal{O}(1) \oplus \mathcal{O}(2)$ & $36$ & $84$ & $-120$ & (b10) & $185$ \\
16 & $$ & $\mathcal{S}^{\ast}(1) \oplus \Lambda^4 \mathcal{Q}$ & $42$ & $84$ & $-98$ & & $27$ \\
17 & $G(2,8)$ & $\Lambda^5 \mathcal{Q} \oplus \mathcal{O}(1)^{\oplus 3}$ & $57$ & $90$ & $-84$ & (b11), VI & $186$ \\
18 \label{sec5} & $$ & $\Sym^2 \mathcal{S}^{\ast} \oplus \Lambda^5 \mathcal{Q}$ & $72$ & $96$ & $-72$ & & $\text{unknown}$ \\
19 & $G(3,6)$ & $\mathcal{O}(1)^{\oplus 6}$ & $42$ & $84$ & $-96$ & (c1), IV & $28$ \\
20 & $$ & $\Lambda^2 \mathcal{S}^{\ast} \oplus \mathcal{O}(1)^{\oplus 2} \oplus \mathcal{O}(2)$ & $32$ & $80$ & $-116$ & (c2) & $42$ \\
21 & $$ & $\mathcal{S}^{\ast}(1) \oplus \Lambda^2 \mathcal{S}^{\ast}$ & $42$ & $84$ & $-96$ & & $28$ \\
22 & $G(3,7)$ & $\Sym^2 \mathcal{S}^{\ast} \oplus \mathcal{O}(1)^{\oplus 3}$ & $128$ & $128$ & $-128$ & (c4) & $3$ \\
23 & $$ & $(\Lambda^2 \mathcal{S}^{\ast})^{\oplus 2} \oplus \mathcal{O}(1)^{\oplus 3}$ & $61$ & $94$ & $-86$ & (c6), VII & $124$ \\
24 & $$ & $(\Lambda^3 \mathcal{Q})^{\oplus 2} \oplus \mathcal{O}(1)$ & $72$ & $96$ & $-74$ & (c3), IX & $\text{unknown}$ \\
25 & $$ & $\Lambda^2 \mathcal{S}^{\ast} \oplus \Lambda^3 \mathcal{Q} \oplus \mathcal{O}(1)^{\oplus 2}$ & $66$ & $96$ & $-84$ & (c5), VIII & $\text{unknown}$ \\
28 & $G(3,8)$ & $(\Lambda^2 \mathcal{S}^{\ast})^{\oplus 4}$ & $92$ & $104$ & $-64$ & & $\text{unknown}$ \\
 \hline
\end{tabular}
\vspace{3mm}
 \caption{The list of 
 direct sums of globally generated irreducible homogeneous vector bundles
 whose zero loci give rise to Calabi-Yau $3$-folds with Picard number one from \cite[Table 1]{IIM}. 
 The numbering is the same as in \cite{IIM}.
 The last column indicates the corresponding numbers of Picard--Fuchs operators in the electronic database \cite{vEvS2}. }
\end{table}

\begin{Main results}
  We determine $I$-functions of Calabi--Yau $3$-folds in Table $\ref{cylist}$ except for No.\ $\dnc$.
  For a Calabi--Yau $3$-fold of No.\ $\dnc$, we give a conjectural $I$-function which is correct if the abelian/nonabelian correspondence holds for determinantal nets of conics. 
  We also determine Picard--Fuchs operators which annihilate these $I$-functions. 
\end{Main results}

Except for No.\ \dnc, we can obtain $J$-functions of Calabi--Yau $3$-folds in Table $\ref{cylist}$ from these $I$-functions. 
If the abelian/nonabelian correspondence holds for determinantal nets of conics, we can also obtain the $J$-function of a Calabi--Yau 3-fold of No.\ \dnc.

\vspace{2mm}
In Appendix \ref{pflist}, we list the resulting differential operators for the $I$-functions of our Calabi--Yau 3-folds with Picard number one. In the references \cite{AvEvSZ}, \cite{vEvS1}, 
hundreds of fourth order differential operators which satisfy the same properties as those for the $I$-functions (called Calabi--Yau operators) are generated. 
It is worth while remarking that some of the differential operators (for No.\ \dnc, 24, 25, 28) are not in their list.
    It is also interesting to observe that some of the obtained Picard--Fuchs operators have more than one maximally unipotent monodromy points, which often indicate non-trivial Fourier--Mukai partners \cite{Rod}, \cite{HT}, \cite{Miu}.

\vspace{2mm}
The organization of this paper is as follows. 
  In Section \ref{prel}, we recall the definitions of $J$-functions and $\mathcal{E}$-twisted $J$-functions. We also introduce $I$-functions for Calabi--Yau $3$-folds. 
In Section \ref{abnonab}, we briefly review the abelian/nonabelian correspondence and apply it to the Calabi--Yau 3-folds in Table \ref{cylist}.
For most of Calabi--Yau 3-folds in Table \ref{cylist}, the $\mathcal{E}$-twisted $I$-functions may be derived readily by applying the abelian/nonabelian correspondence. 
In Section \ref{local}, we deal with vector bundles which contain line bundles as irreducible summands. We determine the quantum differential equations calculating the twisted Gromov--Witten invariants by localization, and applying the quantum Lefschetz theorem. 
In Section \ref{dnc}, we do our calculation for No.\ \dnc \ under the assumption that the abelian/nonabelian correspondence holds for determinantal nets of conics.

\begin{acknowledgements}
We thank Professor Shinobu Hosono and Doctor Fumihiko Sanda 
for valuable discussions 
at weekly seminars and for various useful comments.
A.~I.~was supported by the Grant-in-Aid for JSPS fellows, No.\ 26--1881.
A part of this work was done when M.~M.~was supported by Frontiers of
Mathematical Sciences and Physics at University of Tokyo. 
M.~M.~was also supported by Korea Institute for Advanced Study.
\end{acknowledgements}

\section{Preliminaries}\label{prel}

\subsection{$J$-functions and $I$-functions}
In this subsection, we introduce the $J$-function of a smooth projective variety $X$ over $\mathbb{C}$.
For the sake of simplicity, we shall assume that the cohomology group of $X$ consists of only even degree except when we consider a Calabi--Yau $3$-fold. 
Let $X_{0,n,d}$ be the moduli space of  genus zero, $n$-pointed stable maps to $X$ representing a homology class $d \in H_2(X, \mathbb{Z})$.
Let $\ev_i: X_{0,n,d} \rightarrow X$ be the evaluation map at the $i$-th marked point.
Let $\NE(X) \subset N_1(X)$ be the semigroup generated by the classes of effective curves in $N_1(X)$. 
The Novikov ring of $X$, which is denoted by $\Lambda_{X}$, is defined to be the completion of the semigroup ring $ \mathbb{C}[\NE(X)] $.
We denote by $Q^d \in \Lambda_{X}$ the Novikov variable of an effective curve class $d \in \NE(X)$.
Let $T_0=1, T_1, \dots, T_l, T_{l+1}, \dots, T_m=[\mathrm{pt}]$ be a homogeneous basis of $H^{\ast}(X, \mathbb{C})$ 
so that $T_1, \dots, T_l$ generate $H^2(X, \mathbb{C})$, which determine coordinates $\bm{\tau} = \sum_{i=0}^m \tau_i T_i$ on $H^{\ast}(X, \mathbb{C})$.
Let $T^0, \dots, T^m$ be the dual basis with respect to the Poincar\'e pairing. 
\begin{Definition}
  The $J$-function of $X$, which is a formal function of $\bm{\tau}$ and $1/z$ taking values in $H^{\ast}(X, \mathbb{C}) \otimes \Lambda_X$, is defined as follows: 
  \begin{equation}\label{Jfct}
J_{X}(\bm{\tau}, z) := z + \bm{\tau} + \sum_{d \in H_2(X, \mathbb{Z})} \sum_{a=0}^{m} \sum_{n=0}^{\infty} T^a \frac{Q^d}{n!} \left\langle \bm{\tau}, \dots, \bm{\tau}, \frac{T_a}{z - \psi_{n+1}} \right\rangle_{0,n+1,d},
  \end{equation}
  where
  \begin{equation*}
    \left\langle \bm{\tau}, \dots, \bm{\tau}, \frac{T_a}{z - \psi_{n+1}} \right\rangle_{0,n+1,d} = \int_{[X_{0,n+1,d}]^{\text{vir}}} \ev_1^{\ast} \bm{\tau} \wedge \cdots \wedge \ev_n^{\ast} \bm{\tau} \wedge \frac{\ev_{n+1}^{\ast}T_a}{z - \psi_{n+1}}
  \end{equation*}
  is the genus zero descendant Gromov--Witten invariant of $X$. 
\end{Definition}
Since the genus zero Gromov--Witten theory of $X$ is recovered from the $J$-function of $X$ (see \cite{CG}), the $J$-function is an important object.

We define the small $J$-function of $X$ as the restriction of the $J$-function (\ref{Jfct}) to the small parameter space, i.e.\ $\tau_0 = \tau_{l+1} = \cdots = \tau_{m} = 0$.
We denote $\tau = \sum_{i=1}^{l} \tau_i T_i \in H^2(X, \mathbb{C})$. 

\begin{Example}
  Let $Y$ be a Calabi--Yau $3$-fold of Picard number one.
  Let $H$ be the ample generator of the Picard group of $Y$, which determines the coordinates $\tau H$ on $H^2(Y, \mathbb{C})$.
  Let $\deg Y = H^3$ be the degree of $Y$ with respect to $H$.
  The small $J$-function of $Y$ can be written as follows: 
  \begin{equation*}
    J_{Y}(\tau, z) = ze^{\tau H /z}\left( 1 + \frac{H^2}{\deg (Y)} \sum_{d=1}^{\infty} n_d d^3 \sum_{k=1}^{\infty} \frac{Q^{k d} e^{kd \tau}}{(H + k d z)^2}\right), 
  \end{equation*}
  where $n_d$ is the genus zero Gopakumar--Vafa invariant of $Y$,
  which is related to the genus zero Gromov--Witten invariants of $Y$ by the multiple cover formula \cite{AM}, \cite{GV}.
\end{Example}
\vspace{1mm}

It is difficult to determine the $J$-function from the definition.
To determine the small $J$-functions of nef complete intersection varieties in toric Fano manifolds, Givental introduced small $I$-functions in his proof of mirror theorem.
Recently, $I$-functions are found for broader class of varieties \cite{Iri}, \cite{Bro}, \cite{CFKS}, \cite{CFK}. 
In general, $I$-functions are related to $J$-functions via Birkhoff factorization procedure and change of coordinates (so-called the mirror map).

\vspace{2mm}
In this paper,
we use a terminology,
\textit{small $I$-functions of Calabi--Yau $3$-folds of Picard number one}.
Instead of defining them, we refer to some required properties here:
Let $H$ be the ample generator of the Picard group of a Calabi--Yau 3-fold $Y$.
A small $I$-function of $Y$ is a formal function of $tH \in H^2(Y, \mathbb{C})$ and $1/z$ to $H^{\text{even}}(Y, \mathbb{C})$, and it has the form
\begin{equation*}
I_{Y}(t, z) = I_0(t)z + I_1(t) H + I_2(t) H^2 z^{-1} + I_3(t) H^3 z^{-2}. 
\end{equation*}
This $I$-function is related to the small $J$-function of $Y$ as
\begin{equation}\label{mirrthm}
  J_{Y}(\tau, z) = \frac{I_{Y}(t, z)}{I_0(t)} \quad \left(\tau = \frac{I_1(t)}{I_0(t)}\right).
\end{equation}
The relation is usually called the mirror theorem.

We can give a case-by-case definition of a small $I$-function 
$I_{Y}(t, z)$
of each Calabi--Yau 3-fold $Y$
in Table \ref{cylist} 
by using a twisted $I$-function of an ambient variety,
either Definition \ref{Ifct} or (\ref{qlt}),
together with an appropriate projection and 
identification of
cohomology groups.

\subsection{$\mathcal{E}$-twisted settings}
Let $\mathcal{E}$ be a globally generated vector bundle on $X$ and $Y$ be a zero locus of a general section of $\mathcal{E}$. 
We introduce the $\mathcal{E}$-twisted $J$-function, which is closely related to the $J$-function of $Y$. 
We fix a homogeneous basis $T_i$ $(0 \leq i \leq m)$ of $H^{\ast}(X, \mathbb{C})$ as in the previous subsection.
\begin{Definition}
  Let $\mathcal{E}$ be a globally generated vector bundle on $X$. 
  The $\mathcal{E}$-twisted $J$-function of $X$, which is a function of $\bm{\tau} \in H^{\ast}(X, \mathbb{C})$ and $1/z$ taking values in $H^{\ast}(X, \mathbb{C}) \otimes \Lambda_X$, is defined by
  \begin{equation*}
    J^{\mathcal{E}}_{X}(\bm{\tau}, z) := z + \bm{\tau} + \sum_{d \in H_2(X, \mathbb{Z})} \sum_{a=0}^{m} \sum_{n=0}^{\infty} T^a \frac{Q^d}{n!} \int_{[X_{0,n+1,d}]^{\text{vir}}} e(\mathcal{E}_{0,n,d}') \wedge \ev_1^{\ast} \bm{\tau} \wedge \cdots \wedge \ev_n^{\ast} \bm{\tau} \wedge \frac{\ev_{n+1}^{\ast}T_a}{z - \psi_{n+1}}
  \end{equation*}
  where $\mathcal{E}_{0,n,d}'$ is a suitable vector bundle on $X_{0,n,d}$.
  The precise definition of $\mathcal{E}_{0,n,d}'$ is found in \cite[Section 9]{CG}.
  We define the $\mathcal{E}$-twisted small $J$-function by restricting parameters
  as $\tau_0 = \tau_{l+1} = \cdots = \tau_m = 0$.
\end{Definition}

Let $Y$ be a zero locus of general sections of $\mathcal{E}$ on $X$.
Let $i: Y \hookrightarrow X$ be the natural inclusion.
This inclusion induces a homomorphism of Novikov rings $\Lambda_{Y} \rightarrow \Lambda_{X}$ by $Q^d \mapsto Q^{i_{\ast}d}$ for $d \in H_2(Y, \mathbb{Z})$.
The following fact gives a relation of $J_Y$ and $J_X^{\mathcal{E}}$:
\begin{Fact}[\cite{KKP}]
  Under the above settings, we have
  \begin{equation}\label{CKconj}
    i_{\ast}J_Y(i^{\ast} \bm{\tau}, z) = J_X^{\mathcal{E}}(\bm{\tau}, z) \cup e(\mathcal{E})
  \end{equation}
  where $\bm{\tau} \in H^{\ast}(X, \mathbb{C})$ and the equality holds after applying the above homomorphism $\Lambda_Y \rightarrow \Lambda_X$ to the left hand side. 
\end{Fact}

Since the left hand side of (\ref{CKconj}) involves the push-forward by $i$, we cannot recover full $J$-function of $Y$ in general.
In this paper, we only consider small $J$-functions of Calabi--Yau $3$-folds of Picard number one which are zero loci of homogeneous vector bundles on Grassmannians.
These conditions enable us to determine the small $J$-function $J_Y$ by calculating the $\mathcal{E}$-twisted small $J$-function $J_X^{\mathcal{E}}$.

To apply the equation (\ref{CKconj}), it is convenient to consider a composition of $J_X^{\mathcal{E}}$ with a projection $H^{\ast}(X, \mathbb{C}) \rightarrow H^{\ast}(X, \mathbb{C}) / \ann(e(\mathcal{E}))$, where $\ann(e(\mathcal{E})) = \{x \in H^{\ast}(X, \mathbb{C}) \mid x \cup e(\mathcal{E}) = 0 \}$.
We denote the composite function by $\tilde{J}^{\mathcal{E}}_{X}$.
It is known that $\tilde{J}^{\mathcal{E}}_{X}$ is a function on $H^{\ast}(X, \mathbb{C})/\ann(e(\mathcal{E}))$.

\vspace{2mm}
Similar to $I$-functions, there exist $\mathcal{E}$-twisted $I$-functions for large class of $X$ and $\mathcal{E}$ (cf.\ \cite{CG}, \cite{CFKS}, \cite{CFK}).
In this paper, we mainly use the quantum Lefschetz theorem (see Theorem \ref{qL}) and the abelian/nonabelian correspondence (see Theorem \ref{ItoJ}) to calculate $\mathcal{E}$-twisted $I$-functions. 
As a result, we can calculate $\mathcal{E}$-twisted small $J$-functions from $\mathcal{E}$-twisted small $I$-functions by the same formula as (\ref{mirrthm}).

\section{Quantum differential equations via abelian/nonabelian correspondence}\label{abnonab}

\subsection{Short review of abelian/nonabelian correspondence}
Let
$G$ be a reductive algebraic group. Let $V$ be an affine space with a 
linear action of $G$. By the geometric invariant theory, we define a quasi-projective variety $V/\!\!/G$
from the triple $(V,G,\chi)$, where $\chi$ is a character of $G$. Let $T$ be a maximal torus of $G$ and 
consider the induced action of $T$ on $V$. Associated to $T$, we obtain another variety $V/\!\!/T$ from $(V,T,\chi|_T)$. In this paper, we assume the following properties:
\begin{Assumption}\label{assabnonab}
  Let $V^{\text{s}}(G)$ (resp.\ $V^{\text{ss}}(G)$) be the stable locus (resp.\ semistable locus) of the action of $G$ on $V$.
  We use the same notation for the action of $T$ on $V$. We assume 
  \begin{equation*}
    V^{\text{s}}(G) = V^{\text{ss}}(G), \quad V^{\text{s}}(T) = V^{\text{ss}}(T)
  \end{equation*}
  and $G$ (resp.\ $T$) acts on $V^{\text{s}}(G)$ (resp.\ $V^{\text{s}}(T)$) freely.
  Moreover we assume that both $V^{\text{s}}(G)/G$ and $V^{\text{s}}(T)/T$ are smooth projective varieties and the codimension of $V^{\text{us}}(G) = V \setminus V^{\text{ss}}(G)$ in $V$ is larger than or equal to two.
\end{Assumption}
These properties guarantee that the varieties $V/\!\!/G$ and $V/\!\!/T$ are smooth projective varieties. 
The abelian/nonabelian correspondence is a relation of Gromov--Witten invariants between these two different geometric quotients.
\vspace{2mm}

Let $\Phi$ be a root system corresponding to $(G,T)$.
We fix a decomposition $\Phi = \Phi_{+} \amalg \Phi_-$, where $\Phi_+$ is a set of positive roots and $\Phi_-$ is a set of negative roots. 
We define by $\mathcal{L}_{\alpha}:=V^{\text{s}}(T) \times_T L_{\alpha}$ the line bundle on $V/\!\!/T$ for each root $\alpha \in \Phi$,
where $L_{\alpha}$ is the one-dimensional representation corresponding to $\alpha$.
Let $W := N(T)/T$ be the corresponding Weyl group of $G$,
which acts naturally on $V/\!\!/T$.
We consider subspaces of $H^{\ast}(V/\!\!/T, \mathbb{C})$ defined by 
\begin{align*}
  & H^{\ast}(V/\!\!/T, \mathbb{C})^W = \{ x \in H^{\ast}(V/\!\!/T, \mathbb{C}) \mid \sigma^{\ast}x = x \text{ for all } \sigma \in W \}, \\
  & H^{\ast}(V/\!\!/T, \mathbb{C})^a = \{ x \in H^{\ast}(V/\!\!/T, \mathbb{C}) \mid \sigma^{\ast}x = (-1)^{l(\sigma)} x \text{ for all } \sigma \in W \},
\end{align*}
where $l(\sigma)$ is the length of $\sigma$.

Let $\omega := \prod_{\alpha \in \Phi_{+}} c_1(\mathcal{L}_{\alpha})$.
It is known that $\omega$ is contained in $H^{\ast}(V/\!\!/T, \mathbb{C})^a$ and any element of $H^{\ast}(V/\!\!/T, \mathbb{C})^a$ can be written as a product of $\omega$ and an element of $H^{\ast}(V/\!\!/T, \mathbb{C})^W$.

Next we relate the cohomology of $V/\!\!/G$ $( = V^{\text{s}}(G)/G)$ to that of $V/\!\!/T$ $( = V^{\text{s}}(T)/T)$ using the following diagram:
\[
\begin{CD}
  V/\!\!/T @<i << V^{\text{s}}(G)/T \\
  @. @VV\pi V \\
  @. V/\!\!/G  
\end{CD}. 
\]
In this diagram, the map $i$ is an open embedding and $\pi$ is a $G/T$-fibration.
Note that the action of $W$ on $V/\!\!/T$ stabilizes the open subset $V^{\text{s}}(G)/T$. 
We also note the following fact: 

\begin{Fact}[\cite{Mar}]
  The map $\pi$ induces an isomorphism between $H^{\ast}(V/\!\!/G, \mathbb{C})$ and $H^{\ast}(V^{\text{s}}(G)/T, \mathbb{C})^W$.
  This induces a homomorphism from $H^{\ast}(V/\!\!/T, \mathbb{C})^W$ to $H^{\ast}(V/\!\!/G, \mathbb{C})$ by $(\pi^{\ast})^{-1} \circ i^{\ast}$.
  Moreover the following sequence is exact:
  \begin{equation}\label{exseq}
    0 \longrightarrow \ker (\cup \omega) \longrightarrow H^{\ast}(V/\!\!/T, \mathbb{C})^W \xrightarrow{(\pi^{\ast})^{-1} \circ i^{\ast}} H^{\ast}(V/\!\!/G, \mathbb{C}) \longrightarrow 0,
  \end{equation}
  where $\ker (\cup \omega) = \{x \in H^{\ast}(V/\!\!/T, \mathbb{C})^W \mid x \cup \omega = 0\}$.
\end{Fact}
If $\tilde{\sigma} \in H^{\ast}(V/\!\!/T, \mathbb{C})^W$ and $\sigma \in H^{\ast}(V/\!\!/G, \mathbb{C})$ satisfy
\begin{equation*}
i^{\ast} \tilde{\sigma} = \pi^{\ast} \sigma, 
\end{equation*}
we say that $\tilde{\sigma}$ is a lift of $\sigma$.  
Obviously any $\sigma \in H^{\ast}(V/\!\!/G, \mathbb{C})$ has a lift by (\ref{exseq}). 
Let $T_0 = 1, T_1, \dots, T_{l}, T_{l+1}, \dots, T_{m}$ be a homogeneous basis of $H^{\ast}(V/\!\!/G, \mathbb{C})$, where $T_1,\dots,T_{l}$ generate $H^2(V/\!\!/G, \mathbb{C})$.
In the following, we fix a lift of $T_i$ for each $i$,
and denote it by $\tilde{T}_i$. 
From this, we define the lift of any class $T = \sum_{i=0}^{m} t_i T_i$ by $\tilde{T} = \sum_{i=0}^{m} t_i \tilde{T}_i$.
This defines an injective $\mathbb{C}$-linear map $\iota : H^{\ast}(V/\!\!/G, \mathbb{C}) \rightarrow H^{\ast}(V/\!\!/T, \mathbb{C})^W$. 

Since any $x \in H^{\ast}(V/\!\!/T, \mathbb{C})^a$ can be decomposed to $x = x' \cup \omega$ with $x' \in H^{\ast}(V/\!\!/T, \mathbb{C})^W$, the exact sequence (\ref{exseq}) implies that there exists a unique $\sigma \in H^{\ast}(V/\!\!/G, \mathbb{C})$ such that
\begin{equation}\label{liftnote}
x = \iota (\sigma) \cup \omega. 
\end{equation}
We denote the $\sigma$ above by $\sigma = x/\omega$. 



We denote by $Q^{\tilde{d}} \in \Lambda_{V/\!\!/T}$ (resp. ${\tt Q}^{d} \in \Lambda_{V/\!\!/G}$)
the Novikov variable for $\tilde{d} \in \NE(V/\!\!/T)$ (resp. $d \in \NE(V/\!\!/G)$).
Assumption \ref{assabnonab} implies that there is a unique lifting homomorphism $\Pic(V/\!\!/G) \rightarrow \Pic(V/\!\!/T)^W$.
We also consider the composition $\Pic(V/\!\!/G)_{\mathbb{R}} \rightarrow \Pic(V/\!\!/T)_{\mathbb{R}}^W \hookrightarrow \Pic(V/\!\!/T)_{\mathbb{R}}$. 
Taking the dual of the ample cones in $\Pic(V/\!\!/G)_{\mathbb{R}}$ and $\Pic(V/\!\!/T)_{\mathbb{R}}$ and restricting them to integral classes,
we have the homomorphism of semigroups $h : \NE(V/\!\!/T) \rightarrow \NE(V/\!\!/G)$.
Using this homomorphism, we have the ring homomorphism $p : \Lambda_{V/\!\!/T} \rightarrow \Lambda_{V/\!\!/G}$ 
defined by 
\begin{equation*}
  p(Q^{\tilde{d}}) = (-1)^{\epsilon(\tilde{d})} {\tt Q}^{h(\tilde{d})} \quad (\tilde{d} \in \NE(V/\!\!/T)),
\end{equation*}
where $\epsilon(\tilde{d}) = \int_{\tilde{d}} \sum_{\alpha \in \Phi_+}c_1(\mathcal{L}_{\alpha})$. 

We define an $\mathcal{E}_G$-twisted $I$-function for a globally generated vector bundle $\mathcal{E}_G$ over $V/\!\!/G$ which is induced by a representation of $G$ as follows:
Let $E$ be a representation of $G$ and consider the induced vector bundle $\mathcal{E}_G := V^{\text{s}}(G) \times_G E$ on $V/\!\!/G$. 
Restricting the action of $G$ to $T$, we decompose $E$ into one-dimensional representations $E = \bigoplus_{i=1}^r L_i$ and obtain the corresponding vector bundle $\mathcal{E}_T$ on $V/\!\!/T$.
We note that $\mathcal{E}_T = \bigoplus_{i=1}^r \mathcal{L}_i$ is a direct sum of line bundles. 
Let $I_{V/\!\!/T}^{\mathcal{E}_T}(\tilde{\bm{t}}, z)$ be the $\mathcal{E}_T$-twisted $I$-function of $V/\!\!/T$ and $\mathcal{E}_T$ defined by \cite{CG}, where $\tilde{\bm{t}} \in H^{\ast}(V/\!\!/T, \mathbb{C})$.
We denote by $\partial_{\alpha}$ the vector field on $H^{\ast}(V/\!\!/T, \mathbb{C})$ corresponding to the class $c_1(\mathcal{L}_{\alpha})$ for $\alpha \in \Phi_+$. 
\begin{Definition}[\cite{CFKS}]\label{Ifct}
  We define an $\mathcal{E}_G$-twisted $I$-function by
  \begin{equation}\label{formIfct}
    I_{V/\!\!/G}^{\mathcal{E}_G}(\bm{t}, z) := \frac{1}{\omega}\biggl(\Bigl(\prod_{\alpha \in \Phi_{+}} z \partial_{\alpha}\Bigr) I_{V/\!\!/T}^{\mathcal{E}_T}(\tilde{\bm{t}}, z)\biggr)\bigg|_{Q^{\tilde{d}} = (-1)^{\epsilon(\tilde{d})}{\tt{Q}}^{h(\tilde{d})}, \tilde{\bm{t}}=\iota(\bm{t})}
  \end{equation}
  which is a function of $\bm{t} \in H^{\ast}(V/\!\!/G, \mathbb{C})$.
\end{Definition}

When $V/\!\!/G$ is a partial flag manifold of type $A$, we can obtain the $\mathcal{E}_G$-twisted $J$-function from $I_{V/\!\!/G}^{\mathcal{E}_G}$ by the following theorem:
\begin{Theorem}[\cite{CFKS}]\label{ItoJ}
  Let $V/\!\!/G$ be a partial flag manifold of type A.
  Assume that $\mathcal{E}_G$ is generated by global sections.
  Let $\bm{t} = (t_0, t_1, \dots, t_l, t_{l+1}, \dots, t_{m})$ be the coordinates on $H^{\ast}(V/\!\!/G, \mathbb{C})$ corresponding to the homogeneous basis $T_0,T_1,\dots,T_l,T_{l+1},\dots,T_{m}$ of $H^{\ast}(V/\!\!/G, \mathbb{C})$.
  Then there exists unique $C^i(\bm{t}, z) \in \Lambda_{V/\!\!/G}[\![\bm{t}, z]\!]$ $(i=0, 1, \dots, m)$ by which the $\mathcal{E}_G$-twisted $J$-function $J_{V/\!\!/G}^{\mathcal{E}_G}(\bm{\tau}, z)$ is expressed as
  \begin{equation}\label{anathm}
    J_{V/\!\!/G}^{\mathcal{E}_G}(\bm{\tau}, z) = I_{V/\!\!/G}^{\mathcal{E}_G}(\bm{t}, z) + \sum_{i=0}^{m} C^i(\bm{t}, z) z \partial_{t_i}I_{V/\!\!/G}^{\mathcal{E}_G}(\bm{t},z).
  \end{equation}
  The coordinate change $\bm{\tau} = \bm{\tau}(\bm{t})$ can be read off from the coefficient of $z^0$ in the $z$-expansion of the right hand side.
\end{Theorem}

  We will apply this theorem to Grassmannians $V/\!\!/G = G(k, n)$ and globally generated vector bundles $\mathcal{E}_G$ which are induced by representations and define Calabi--Yau $3$-folds $Y$ of Picard number one.

  In the following calculation, we restrict the parameter $\bm{t}$ to the small parameter space $t \in H^2(V/\!\!/G, \mathbb{C})$.
  In the small parameter space, we can express the $\mathcal{E}_T$-twisted $I$-function by combinatorial data of the toric variety $V/\!\!/T$ and the vector bundle $\mathcal{E}_{T}$ from \cite{Giv2}.
  It is known (see \cite[Section D]{CCGK}) that $I_{V/\!\!/G}^{\mathcal{E}_G}(t, z)$ becomes homogeneous of degree one under the following  degrees:
  \begin{equation*}
    \deg z = 1,\quad \deg T_i = a \quad \text{when } T_i \in H^{2a}(V/\!\!/G, \mathbb{C}).
  \end{equation*}
  Let $H$ be the class of the Schubert divisor on $G(k,n)$. 
  Then we can expand $ I_{V/\!\!/G}^{\mathcal{E}_G}(t, z) $ as 
  \begin{equation*}
    I_{V/\!\!/G}^{\mathcal{E}_G}(t, z) = I_0(t) z + I_1(t) H + O(z^{-1}). 
  \end{equation*}
  By the property $z \partial_{t_0} I_{V/\!\!/G}^{\mathcal{E}_G}(\bm{t}, z) = I_{V/\!\!/G}^{\mathcal{E}_G}(\bm{t}, z)$ and the uniqueness of $C^i(\bm{t}, z)$ $(i = 0, \dots, m)$, 
  the transformation (\ref{anathm}) in small parameter space can be written as  
  \begin{equation*}
    J_{V/\!\!/G}^{\mathcal{E}_G}(\tau, z) = \frac{I_{V/\!\!/G}^{\mathcal{E}_G}(t, z)}{I_0(t)},
  \end{equation*}
  where the coordinate change $\tau = (I_1(t) /I_0(t)) H$ is called the mirror map.

  Similar to $\tilde{J}^{\mathcal{E}_G}_{V/\!\!/G}$, we will denote by $\tilde{I}_{V/\!\!/G}^{\mathcal{E}_G}$ the composition of $I_{V/\!\!/G}^{\mathcal{E}_G}$ with the natural projection $H^{\ast}(V/\!\!/G, \mathbb{C}) \rightarrow H^{\ast}(V/\!\!/G, \mathbb{C}) / \ann(e(\mathcal{E}_G))$.

\subsection{$I$-functions and quantum differential equations}\label{subsec_I&QDE}
Let $X = G(k,n)$ be the Grassmannian of $k$-planes in $\mathbb{C}^n$. We describe $X$ by the geometric invariant theory quotient $V/\!\!/G$ with $V = \Mat_{k \times n}(\mathbb{C})$, $G = GL(k, \mathbb{C})$ and a suitable choice of $\chi$. 
It is well-known that any irreducible representation of $GL(k, \mathbb{C})$ is given by applying the Schur functor to the standard representation.
The tautological subbundle $\mathcal{S}$ is a vector bundle over $G(k,n)$ associated to the standard representation of $GL(k, \mathbb{C})$,
and hence the dual $\mathcal{S}^*$ corresponds to the dual of the standard representation.
A globally generated vector bundle associated to an irreducible representation of $G$ is described as $\Sigma^{\lambda}\mathcal{S}^{\ast}$, where $\Sigma^{\lambda}$ is the Schur functor with $\lambda = (\lambda_1, \lambda_2, \dots, \lambda_k)$ ($\lambda_1 \geq \lambda_2 \geq \cdots \geq \lambda_k \geq 0$).
The abelian/nonabelian correspondence applies to the twisted $I$-functions with vector bundles of the form
\begin{equation}\label{sdvbdl}
 \mathcal{E} = \bigoplus_{i=1}^r \Sigma^{\lambda_i}\mathcal{S}^{\ast},
\end{equation}
where $\lambda_i = (\lambda_{i,1},\lambda_{i,2},\dots, \lambda_{i,k})$ for $i = 1,2, \dots, r$. 
Now let us collect necessary data for our calculations of $I$-functions. 

\begin{enumerate}
\item[(i)](Associated toric variety) Let $T = (\mathbb{C}^{\ast})^k$ be the maximal torus of $GL(k,\mathbb{C})$ consisting of diagonal matrices.
  Then $V/\!\!/T = \underbrace{\mathbb{P}^{n-1} \times \cdots \times \mathbb{P}^{n-1}}_{k}$. 
\item[(ii)](The fundamental Weyl anti-invariant class) Let $W = \mathfrak{S}_k$ be the Weyl group of $GL(k, \mathbb{C})$.
  It acts on $V/\!\!/T= \mathbb{P}^{n-1} \times \cdots \times \mathbb{P}^{n-1}$ by permuting $k$ factors. 
  The anti-invariant class $\omega$ is given by $\prod_{1 \leq i < j \leq k} (H_i - H_j)$,
  where $H_i$ is the pullback of the hyperplane class of the $i$-th factor. 
\item[(iii)](Restriction of the coordinates) Let $H$ be the class of the Schubert divisor of $G(k,n)$.
  The Weyl invariant lift of $H$ is given by $H_1 + \cdots + H_k$.
  We have an isomorphism $H^{2}(V/\!\!/G, \mathbb{C}) \cong H^{2}(V/\!\!/T, \mathbb{C})^W$ which maps $tH \in H^2(V/\!\!/G,\mathbb{C})$ to $t(H_1 + \cdots + H_k) \in H^{2}(V/\!\!/T, \mathbb{C})^W$.
\item[(iv)](Novikov rings) Let $C_1, \dots, C_k$ be the dual basis of $H_1, \dots, H_k$ with respect to the pairing between $H_2(V/\!\!/T, \mathbb{C})$ and $H^2(V/\!\!/T, \mathbb{C})$.
  We denote by $Q_i = Q^{C_i}$.
  Then the Novikov ring of $V/\!\!/T$ is isomorphic to $\mathbb{C}[\![Q_1,\dots, Q_k]\!]$.
  Similarly we have $\Lambda_{V/\!\!/G} = \mathbb{C}[\![{\tt{Q}}]\!]$ for the Novikov ring of $V/\!\!/G$.
  Note that $\Lambda_{V/\!\!/G}$ is generated by a single generator ${\tt{Q}}$ since $V/\!\!/G$ has Picard number one. 
  Then the ring homomorphism $p : \Lambda_{V/\!\!/T} \rightarrow \Lambda_{V/\!\!/G}$ is given by $p(Q_i) = (-1)^{k-1}{\tt{Q}}$.
\item[(v)](Decomposition of vector bundles) The irreducible vector bundle $\mathcal{E}_G = \Sigma^{\lambda}\mathcal{S}^{\ast}$ on $X$ decomposes into direct sums of line bundles $\mathcal{E}_T$ determined by weights of the irreducible representation of $G$.
  For example if $\mathcal{E}_G = \mathcal{S}^{\ast}$, then $\mathcal{E}_T = \bigoplus_{i=1}^k \mathcal{O}(0,\dots,0,\overset{i}{\check{1}},0,\dots,0)$.
\end{enumerate}

\vspace{4mm}
With the above data, we can apply the abelian/nonabelian correspondence to the vector bundle of the form (\ref{sdvbdl}) over $G(k,n)$.
For vector bundles of the form 
\begin{equation}\label{qvbdl}
  \mathcal{E} = \bigoplus \Sigma^{\lambda_i} \mathcal{Q},
\end{equation}
i.e.\ those constructed by universal quotient bundle $\mathcal{Q}$, 
we can use the duality $G(k,n) \cong G(n-k,n)$ to reduce (\ref{qvbdl}) to the case (\ref{sdvbdl}) above. 

Using either the representation (\ref{sdvbdl}) or (\ref{qvbdl}), we can calculate the twisted $I$-functions for most of Calabi--Yau 3-folds in Table \ref{cylist} except for the following three cases
\begin{align*}
  &\text{No}.\ \linseq: \mathcal{S}^{\ast}(1) \oplus \wedge^4 \mathcal{Q} \text{ on } G(2,7),\\
  &\text{No}.\ \dnc: \Sym^2 \mathcal{S}^{\ast} \oplus \wedge^5 \mathcal{Q} \text{ on } G(2,8),\\
  &\text{No}.\ \sqmix: \wedge^2 \mathcal{S}^{\ast} \oplus \wedge^3 \mathcal{Q} \oplus \mathcal{O}(1)^{\oplus 2} \text{ on } G(3,7).
\end{align*}
Since a Calabi--Yau $3$-fold of No.\ $\linseq$ is a deformation equivalent to a linear section Calabi--Yau $3$-fold in $G(2,7)$ by \cite[Proposition 5.1]{IIM}, the calculation of the first case reduces to the calculation of No.\ 12.
We can calculate the twisted $I$-function of a Calabi--Yau $3$-fold of No.\ 12 by the abelian/nonabelian correspondence.
We will consider the last two cases in the subsequent sections.
  
For the rest of Calabi--Yau 3-folds in Table \ref{cylist}, it is straightforward to determine the $I$-functions.
Since there is no difference in the calculations, we only present the details for a selected example.
\begin{Example}\label{exana}
  Consider a Calabi--Yau 3-fold $Y$ of type No.\ 7 in Table \ref{cylist}.
  It is given by the zero locus of a general section of $\mathcal{E}_G = \mathcal{S}^{\ast}(1) \oplus \mathcal{O}(1)^{\oplus 3}$ on $V/\!\!/G = G(2,6)$.
  Let $T = (\mathbb{C}^{\ast})^2$ be a maximal torus of $GL(2, \mathbb{C})$.
  Then the abelian quotient is $V/\!\!/T = \mathbb{P}^5 \times \mathbb{P}^5$.
  The decomposition of $\mathcal{E}_T$ is given by $\mathcal{E}_T = \mathcal{O}(2,1) \oplus \mathcal{O}(1,2) \oplus \mathcal{O}(1,1)^{\oplus 3}$.
  For each Schubert class $\sigma \in H^{\ast}(G(2,6), \mathbb{C})$, we fix a lift of it by the corresponding Schur polynomial of $H_1,H_2$, where $H_i$ is the hyperplane class of the $i$-th component.

  \vspace{2mm}
  Step 1.
  We introduce affine coordinates $\tilde{t}_1, \tilde{t}_2$ which correspond to the basis $H_1, H_2 \in H^{2}(\mathbb{P}^5 \times \mathbb{P}^5, \mathbb{C})$.
  The $\mathcal{E}_T$-twisted small $I$-function is given by the following formula (see \cite{Giv2}):
 \begin{equation*}
   I_{\mathbb{P}^5 \times \mathbb{P}^5}^{\mathcal{E}_T}(\tilde{t}_1, \tilde{t}_2, z) = z e^{(\tilde{t}_1 H_1 + \tilde{t}_2 H_2)/z} \sum_{d_1, d_2 \geq 0} I_{(d_1,d_2)}^{\mathcal{E}_T}e^{d_1 \tilde{t}_1} e^{d_2 \tilde{t}_2}Q_1^{d_1}Q_2^{d_2}, 
 \end{equation*}
 where 
 \begin{equation*}
   I_{(d_1,d_2)}^{\mathcal{E}_T}=\frac{\prod_{m=1}^{2d_1 + d_2}(2H_1 + H_2 + m z)
   \prod_{m=1}^{d_1 + 2d_2}(H_1 + 2H_2 + mz) \prod_{m=1}^{d_1 + d_2}(H_1
   + H_2 + mz)^3}{\prod_{m=1}^{d_1}(H_1 + m z)^6\prod_{m=1}^{d_2}(H_2 + m z)^6}.
 \end{equation*}
 We also introduce an affine coordinate $t$ in $H^2(G(2,6), \mathbb{C})$ which corresponds to the class of the Schubert divisor $H$. By Definition \ref{Ifct}, we have 
 \begin{equation*}
  I_{G(2,6)}^{\mathcal{E}_G}(t,z) = \frac{1}{H_1-H_2} z e^{t (H_1+H_2)/z} \sum_{d = 0}^{\infty} (-1)^{d}{\tt{Q}}^{d}q^d \sum_{\substack{d_1 + d_2 = d\\d_1 ,d_2 \geq 0}}(H_1 - H_2 + z(d_1 - d_2)) I_{(d_1,d_2)}^{\mathcal{E}_T}, 
 \end{equation*}
 where $q = e^t$.

 \vspace{2mm}
 Step 2.
 We expand $\tilde{I}_{G(2,6)}^{\mathcal{E}_G}(t,z)$ with respect to $z$.
 Then we obtain the expansion of the form $\tilde{I}_{G(2,6)}^{\mathcal{E}_G}(t,z) = I_0(t) z + I_1(t) H + I_2(t) H^2 z^{-1} + I_3(t) H^3 z^{-2}$, where $H$ is the class of the Schubert divisor of $G(2,6)$ in $H^{\ast}(G(2,6), \mathbb{C}) / \ann(e(\mathcal{E}_G))$.
 We set ${\tt{Q}} = 1$.
 The functions $I_0(t)$, $I_1(t)$ are given by
 \begin{align*}
   I_0(t) &= 1 + 7q + 199q^2 + 8359q^3 + \cdots, \\
   I_1(t) &= I_0(t) t + 21q + \frac{1431}{2}q^2 + \frac{64373}{2}q^3 + \cdots.
 \end{align*}
 The function $I_0$ is expected to coincide with a period integral of a mirror family $\{Y_q^{\vee}\}_{q \in \mathbb{P}^1}$ for a monodromy invariant cycle $\gamma$ around the so-called maximally unipotent monodromy point.
 We look for a differential operator of fourth order which annihilates $\tilde{I}_{G(2,6)}^{\mathcal{E}_G}$ assuming the following general form of the operator with polynomial coefficients of degree at most $d$, 
 \begin{equation*}
   P = \sum_{i=0}^4 a_i(q) \theta^i, \; a_i(q) = \sum_{j=0}^d a_{i,j}q^j, 
 \end{equation*}
 where $a_{i,j}$ are unknown constants and $\theta = q \frac{d}{dq}$. Applying $P$ to $I_0$, we obtain sufficient numbers of linear equations $\{L_m(a_{ij})\}$ for $a_{ij}$.
 Since we can calculate $I_0$ for arbitrary degree, we can find the differential operator $P$ which annihilates $I_0$ (see Appendix \ref{pflist}). 
\end{Example}

\section{Quantum differential equations via quantum Lefschetz theorem}\label{local}
As explained in Subsection \ref{subsec_I&QDE}, we cannot apply the abelian/nonabelian correspondence to the Calabi--Yau 3-folds of type No.\ \dnc \ and No.\ \sqmix \ in Table \ref{cylist}. 
To determine an $\mathcal{E}$-twisted $I$-function for No.\ \sqmix,
we will take another approach which is similar to Tj\o tta's work \cite{Tjo2}. 
This approach can be applied to the homogeneous vector bundles $\mathcal{E}$ in Table \ref{cylist} 
which splits into $\mathcal{E} = \mathcal{E}' \oplus \mathcal{H}$, where
$\mathcal{E}'$ does not contain line bundles in its direct summand and $\mathcal{H} =  \bigoplus_{i=1}^r \mathcal{O}(d_i)$ $(r \geq 1)$. 
\subsection{Quantum Lefschetz theorem}
Let us recall the following theorem:
\begin{Theorem}[Quantum Lefschetz theorem \cite{Kim}, \cite{Lee}, \cite{CG}]\label{qL}
  Let $X$ be a smooth projective variety and $\mathcal{E}'$ be a holomorphic vector bundle on $X$.
  Let $\mathcal{H} = \bigoplus_{i=1}^r \mathcal{L}_i$ be a direct sum of line bundles.
  We assume that both $\mathcal{E}'$ and $\mathcal{H}$ are globally generated.
  Let $J_{X}^{\mathcal{E}'}(\bm{t}, z) = \sum_{d \in \NE(X)} J_d(\bm{t}, z) Q^d$ be the $\mathcal{E}'$-twisted $J$-function.
  We define
 \begin{equation}\label{qlt} 
   I^{\mathcal{H}}_{\mathcal{E}'}(\bm{t}, z) := \sum_{d \in \NE(X)} \prod_{i=1}^r \prod_{m=1}^{d . c_1(\mathcal{L}_i)}(c_1(\mathcal{L}_i) + mz)J_d(\bm{t}, z)Q^d
 \end{equation}
 as a modification of $J_{X}^{\mathcal{E}'}$.
 Then $J_{X}^{\mathcal{E}' \oplus \mathcal{H}}(\bm{\tau}, -z)$ and $I_{\mathcal{E}'}^{\mathcal{H}}(\bm{t}, -z)$ generate the same Lagrangian cone in the semi-infinite symplectic space $(H(\!(z^{-1})\!), \Omega)$ for $H = H^{\ast}(X, \mathbb{C})$ and a suitable symplectic form $\Omega$.
 In particular, the same relation as $(\ref{anathm})$ holds between $J_{X}^{\mathcal{E}' \oplus \mathcal{H}}$ and $I_{\mathcal{E}'}^{\mathcal{H}}$.
\end{Theorem}

We will apply this theorem to many examples in Table \ref{cylist}, 
whose vector bundles on $X=G(k,n)$
are of the form $\mathcal{E}'\oplus \mathcal{H}$ with 
non-zero $\mathcal{H} = \bigoplus_{i=1}^r \mathcal{O}(d_i)$.
We define $\tilde{I}_{\mathcal{E}'}^{\mathcal{H}}$ as the composite function of $I_{\mathcal{E}'}^{\mathcal{H}}$ and the natural projection $H^{\ast}(X, \mathbb{C}) \rightarrow H^{\ast}(X, \mathbb{C})/\ann(e(\mathcal{E}' \oplus \mathcal{H}))$.
We consider the restriction of $\bm{t} \in H^{\ast}(X, \mathbb{C})$ to the second cohomology $t \in H^2(X, \mathbb{C})$.
We set $z = 1$ and $Q^d = 1$ for $d \in \NE(X)$.
Our purpose is to find a fourth order differential operator which annihilates $\tilde{I}_{\mathcal{E}'}^{\mathcal{H}}(t)$.
Suppose we have a differential operator $Q(q, \theta)$ which annihilates $\tilde{J}_{X}^{\mathcal{E}'}(t)$.
Since $\dim H^2(X, \mathbb{C}) = 1$ for $X = G(k,n)$, we can use the following result (\cite{BvS}) to determine the differential operator which annihilates $\tilde{I}_{\mathcal{E}'}^{\mathcal{H}}$ from $Q(q, \theta)$:
\begin{Lemma}[\cite{BvS}]\label{BvS}
  Let us write $Q(q, \theta) = \sum_{j=0}^d q^j Q_j(\theta)$ for $Q_j(\theta) \in \mathbb{Q}[\theta]$.
  Then the differential operator
  \begin{equation*}
    P(q, \theta) := \sum_{j=0}^d q^j Q_j(\theta) \prod_{i=1}^{r} \prod_{m=1}^{j} (d_i \theta + m)
  \end{equation*}
  annihilates $\tilde{I}_{\mathcal{E}'}^{\mathcal{H}}(t)$. 
\end{Lemma}

Considering the relation between twisted $J$-functions and quantum differential systems, we can obtain $Q(q, \theta)$ explicitly. 
The details are described in the following subsections. 

Here we remark that the rank of $P$ is larger than four in general. However we can extract a Picard--Fuchs operator from $P$ after suitable factorization of the operator.

\subsection{$J$-function and quantum differential system}
We characterize the $J$-function as a collection of flat sections of the local system associated to the quantum cohomology.
The same formula holds for $\tilde{J}^{\mathcal{E}'}_X$ if we replace Gromov--Witten invariants, quantum product and $H^{\ast}(X, \mathbb{C})$ with $\mathcal{E}'$-twisted Gromov--Witten invariants, $\mathcal{E}'$-twisted quantum product and $H^{\ast}(X, \mathbb{C})/ \ann(e(\mathcal{E}'))$ respectively.


\vspace{2mm}
Let us consider $\mathcal{M} = H^{\ast}(X, \mathbb{C})$ as an affine space $\mathbb{C}^{m+1}$ with coordinates $t_0, t_1, \dots, t_l, t_{l+1}, \dots, t_m$ by using the basis $T_0, T_1, \dots, T_l, T_{l+1}, \dots, T_m$ as in previous sections. 
Then the tangent bundle $\mathcal{T}_{\mathcal{M}}$ is a trivial bundle with fiber $H^{\ast}(X, \mathbb{C})$.
Using the quantum product, we can introduce the connection $\nabla^z$ on $\mathcal{T}_{\mathcal{M}}$ by 
\begin{equation*}
  \nabla^z_{\frac{\partial}{\partial t_i}}T_j = z^{-1} T_i \ast T_j \hspace{3mm} (i = 0,1, \dots, m).
\end{equation*}
Due to the associativity of the quantum product, it turns out that the connection $\nabla^z$ is flat and defines the corresponding flat sections $s_i(\bm{t})$ on $\mathcal{M}$. 
The following proposition is a classical results (see \cite{CK}): 
\begin{Proposition}
  The flat sections $s_i(\bm{t})$ $(0 \leq i \leq m)$ are given by  
  \begin{equation*}
    s_i(\bm{t}) = T_i + \sum_{n \geq 0}\sum_{d \in \NE(X)}\sum_{j=0}^m \frac{Q^d}{n!}\left\langle \frac{T_i}{z - \psi},T_j,\underbrace{\bm{t},\dots,\bm{t}}_{n} \right\rangle_{0,n+2,d}T^j.
  \end{equation*}
  These sections determine the $J$-function by
  \begin{equation*}
    J(\bm{t}, z) = z \sum_{i=0}^m \langle s_i, 1 \rangle T^i, 
  \end{equation*}
  where $\langle , \rangle$ is the Poincar\'e pairing on $H^{\ast}(X, \mathbb{C})$.
\end{Proposition}

For our purpose, it is sufficient to calculate the small $J$-function.
Therefore we restrict the coordinate $\bm{t}$ to $H^2(X, \mathbb{C})$ by setting $t_0 = t_{l+1} = \cdots = t_m = 0$.
Then the quantum connection $\nabla^z$ is determined by the small quantum product for divisor classes $T_i$ as follows:
\begin{equation*}
\nabla^z_{\frac{\partial}{\partial t_i}} T_j = z^{-1} T_i \ast_{\text{small}} T_j \hspace{3mm} (i = 1,\dots, l). 
\end{equation*}
In the following, we denote the small quantum product by $\ast$ for simplicity.
Since we do not use Novikov variables in the following argument, we set $Q^d = 1$ for simplicity. 

\subsection{Calculation of the quantum differential equation of $\tilde{J}_{X}^{\mathcal{E}'}$}\label{genmet}
Let $X = G(k,n)$ and $\mathcal{E}'$ be as above. We fix a homogeneous basis $T_0,T_1,\dots,T_m$ of $H^{\ast}(X, \mathbb{C})/\ann(e(\mathcal{E}'))$ such that $T_0 = 1$, $T_1$ is the class of the Schubert divisor of $X$, and $T_m$ has top degree $2( \dim X - \mathrm{rank} \ \mathcal{E}')$.
Let $T^0, T^1, \dots, T^m$ be the dual basis with respect to the $\mathcal{E}'$-twisted Poincar\'e pairing, i.e.\ $(T_i, T^j)_{e(\mathcal{E}')} = \int_{X} T_i \cup T^j \cup e(\mathcal{E}') = \delta_i^j$. 
By the definition of the $\mathcal{E}'$-twisted small quantum product, we have 
\begin{equation*}
  T_i \ast_{\mathcal{E}'} T_j = \sum_{k=0}^m \sum_{d \geq 0} \langle T_i,T_j,T_k \rangle_{0,3,d}^{\mathcal{E}'} e^{d t} T^k. 
\end{equation*}
Note that the correlators ($\mathcal{E}'$-twisted Gromov--Witten invariants) $\langle T_i,T_j,T_k \rangle_{0,3,d}^{\mathcal{E}'}$ can be nonvanishing only if the degree equality
\begin{equation}\label{constra}
  \deg T_i + \deg T_j + \deg T_k = \dim X - \text{rank } \mathcal{E}' + (c_1(X)-c_1(\mathcal{E}')) d
\end{equation}
holds.
Since $c_1(X) - c_1(\mathcal{E}') > 0$ in our settings of this section, the degrees $d$ which satisfy (\ref{constra}) are finite.
This means that the $\mathcal{E}'$-twisted small quantum product is determined by a finite number of $\mathcal{E}'$-twisted Gromov--Witten invariants. 

In order to calculate the $\mathcal{E}'$-twisted Gromov--Witten invariants, we will use the action of a maximal torus $T = (\mathbb{C}^{\ast})^n$ on $G(k,n)$. This torus action is large enough to calculate $\mathcal{E}'$-twisted Gromov--Witten invariants by the localization formula \cite{Kon} due to Kontsevich. 
\bigskip

Let $M = (M_{ij})$ be the connection matrix of $\nabla^z$ with respect to the fixed basis, i.e.\
\begin{equation}\label{qdsystem}
 \nabla^z_{\frac{d}{dt}}T_j = z^{-1} T_1 \ast_{\mathcal{E}'} T_j = z^{-1} \sum_{i=0}^m M_{ij}T_i. 
\end{equation}
  We set $z=1$.
The condition that a section $s = \sum_{i=0}^m s_i(t)T_i$ is flat is equivalent to 
\begin{equation*}
  \sum_{i=0}^m \frac{ds_i(t)}{dt}T_i = T_1 \ast_{\mathcal{E}'} \sum_{i=0}^m s_i(t)T_i, 
\end{equation*}
which entails the system of linear differential equations 
\begin{equation}\label{odesystem}
 \frac{ds_i(t)}{dt} = \sum_{j=0}^m M_{ij}s_j(t) \hspace{3mm} (i = 0,1, \dots, m). 
\end{equation}

In order to find a differential operator which annihilates $\tilde{J}^{\mathcal{E}'}_{X}$, it is enough to find a differential operator which annihilates $s_m(t)$ for any flat section $s = \sum_{i=0}^m s_i(t)T_i$.
Using (\ref{odesystem}) repeatedly, we express the higher derivatives $d^j s_m(t)/dt^j$ $(j = 0, 1, \dots, m+1)$ of $s_m(t)$ by suitable linear combinations of $s_0(t), \dots, s_m(t)$ as
\begin{equation}\label{hod}
  \frac{d^j s_m(t)}{dt^j} = \sum_{i=0}^m C_{ij}(q) s_i(t),
\end{equation}
where $C = (C_{ij})$ is a matrix of size $(m+1) \times (m+2)$ with entries in $\mathbb{C}[q]$.
This matrix has the kernel 
\begin{equation*}
  {}^t(-\det C_1, \det C_2, \dots, (-1)^{m+2}\det C_{m+2}) =: {}^t (f_0(q),f_1(q), \dots, f_{m+1}(q)),
\end{equation*}
where $C_i$ is the matrix obtained from $C$ by deleting the $i$-th column.
Evaluating the both sides of (\ref{hod}) with this kernel, we obtain
\begin{equation}\label{deriv}
\left(\sum_{i=0}^{m+1} f_i(q) \frac{d^i}{dt^i} \right)s_m(t) = 0. 
\end{equation}
This is the operator $Q(q, \theta)$ which annihilates $\tilde{J}^{\mathcal{E}'}_{X}$.
 
 \subsection{The case of $\mathcal{E} = \wedge^2 \mathcal{S}^{\ast} \oplus \wedge^3 \mathcal{Q} \oplus \mathcal{O}(1)^{\oplus 2}$}
 Here we apply the method described above for Calabi--Yau $3$-folds in
Table \ref{cylist} of the vector bundle $\mathcal{E} = \mathcal{E}' \oplus \mathcal{H}$. Although we only present the detail of the calculations for No.\ \sqmix \ in Table \ref{cylist}, the same method can be applied to other cases, too. 
 \bigskip
 
 A Calabi--Yau $3$-fold of No.\ \sqmix \ is given by the zero locus of a general section of $\mathcal{E} = \wedge^2 \mathcal{S}^{\ast} \oplus \wedge^3 \mathcal{Q} \oplus \mathcal{O}(1)^{\oplus 2}$ on $X = G(3,7)$.
 Let us denote $\mathcal{E}' = \wedge^2 \mathcal{S}^{\ast} \oplus \wedge^3 \mathcal{Q}$.
 We will determine all two point $\mathcal{E}'$-twisted Gromov--Witten invariants.
 Let $T = (\mathbb{C}^{\ast})^7$ be the maximal torus which naturally acts on $G(3,7)$.
 We use this $T$-action to compute $\mathcal{E}'$-twisted Gromov--Witten invariants. 

 Let $s_i=c_i(\mathcal{Q})$ be the $i$-th Chern class of $\mathcal{Q}$ for $1 \leq i \leq 4$.
 After some algebra, it is easy to see that we can take a basis of the $H^{\ast}(X, \mathbb{C})/\ann(e(\mathcal{E}'))$ by
 \begin{equation}\label{basis}
   T_0,T_1,\dots,T_7 = 1,s_1,s_1^2,s_2,s_1^3,s_1 s_2,s_1^4,s_1^5.
 \end{equation}
 We can calculate the twisted Poincar\'{e} pairing $(\alpha, \beta)_{e(\mathcal{E}')} = \int_{G(3,7)} \alpha \cup \beta \cup e(\mathcal{E}')$ as follows:
 \begin{align*}
   &\int_{G(3,7)} s_1^5 e(\mathcal{E}') = 66,& &\int_{G(3,7)} s_1^3 s_2 e(\mathcal{E}') = 36,& &\int_{G(3,7)} s_1 s_2^2 e(\mathcal{E}') = 20.&
 \end{align*}
 Let $T^0, \dots, T^7$ be the dual basis of $T_0, \dots, T_7$ with respect to $(,)_{e(\mathcal{E}')}$.
 The $\mathcal{E}'$-twisted small quantum product $\ast_{\mathcal{E}'}$ is defined by 
 \begin{equation*}
   T_i \ast_{\mathcal{E}'} T_j = \sum_{k=0}^7\sum_{d \geq 0} \langle T_i,T_j,T_k \rangle_{0,3,d}^{\mathcal{E}'} q^d T^k.
 \end{equation*}
 To determine the quantum connection, we need to calculate the $\mathcal{E}'$-twisted small quantum product with divisor class $T_1$.
 Using the properties of twisted Gromov--Witten invariants (see \cite{CK}), it turns out that it is enough to calculate the following two point $\mathcal{E}'$-twisted Gromov--Witten invariants $\langle T_i, T_j \rangle_{0,2,d}^{\mathcal{E}'}$:
 \begin{Lemma}\label{loc}
   For $\mathcal{E}'$-twisted Gromov--Witten invariants, we have
   \begin{align*}
     & \langle s_1^5 \rangle_{0,1,1}^{\mathcal{E}'} = 264, & & \langle s_1^2,s_1^4 \rangle_{0,2,1}^{\mathcal{E}'} = 594, & & \langle s_2,s_1^4 \rangle_{0,2,1}^{\mathcal{E}'} = 330, &  \\
     & \langle s_1^3,s_1^3 \rangle_{0,2,1}^{\mathcal{E}'} = 744, &  & \langle s_1^3,s_1 s_2 \rangle_{0,2,1}^{\mathcal{E}'} = 408, & & \langle s_1 s_2, s_1 s_2 \rangle_{0,2,1}^{\mathcal{E}'} = 224, & \\
     & \langle s_2,s_2,s_1 s_2 \rangle_{0,3,1}^{\mathcal{E}'} = 176, & & \langle s_1^3, s_1^5 \rangle_{0,2,2}^{\mathcal{E}'} = 2376, &  & \langle s_1^4, s_1^4 \rangle_{0,2,2}^{\mathcal{E}'} = 4356, & \\
     & \langle s_1^5, s_1 s_2 \rangle_{0,2,2}^{\mathcal{E}'} = 1320. 
   \end{align*}
 \end{Lemma}
 \begin{proof}
   Since $\mathcal{E}'$ is an equivariant vector bundle with respect to the $T$-action on $G(3,7)$, we can calculate these numbers by using the localization formula given by \cite{Kon}. We refer \cite{Ino} for the details of this calculation.
 \end{proof}
  
 \begin{Lemma}\label{allGW}
   Any three point $\mathcal{E}'$-twisted Gromov--Witten invariants are determined uniquely from the numbers in Lemma $\ref{loc}$ by WDVV relations.
 \end{Lemma}
 \begin{proof}
   Recall that $\mathcal{E}'$-twisted Gromov--Witten invariants satisfy WDVV relations, which are quadratic relations among them.
   More precisely for any $0 \leq i, j, k, l \leq 7$ and $d \geq 0$, we have
   \begin{equation*}
     \sum_{m=0}^d \sum_{a=0}^7 \langle T_i, T_j, T_a \rangle_{0,3,m}^{\mathcal{E}'} \langle T^a, T_k, T_{l} \rangle_{0,3,d-m}^{\mathcal{E}'} = \sum_{m=0}^d \sum_{a=0}^7 \langle T_i, T_{l}, T_a \rangle_{0,3,m}^{\mathcal{E}'} \langle T^a, T_j, T_k \rangle_{0,3,d-m}^{\mathcal{E}'}.
   \end{equation*}
   We can verify directly that WDVV relations determine any other three point $\mathcal{E}'$-twisted Gromov--Witten invariants if we are given the numbers in Lemma \ref{loc}.
 \end{proof}

  \begin{Proposition}\label{deribeqde}
    The connection matrix of $\nabla^z_{\frac{d}{dt}}$ with respect to the basis $(\ref{basis})$ is given by 
    \begin{equation*}
      \nabla^z_{\frac{d}{dt}} = \frac{d}{dt} - \frac{1}{z}
      \left(\begin{smallmatrix}
        0 & 4q & 0 & 0 & 72q^2 & 40q^2 & 0 & 396q^3 \\
        1 & 0 & 9q & 5q & 0 & 0 & 132q^2 & 0 \\
        0 & 1 & 0 & 0 & 8q & 4q & 0 & 0 \\
        0 & 0 & 0 & 0 & 6q & 4q & 0 & 132q^2 \\
        0 & 0 & 1 & 0 & 0 & 0 & 0 & 0 \\
        0 & 0 & 0 & 1 & 0 & 0 & \frac{33}{2}q & 0 \\
        0 & 0 & 0 & 0 & 1 & \frac{6}{11} & 0 & 4q \\
        0 & 0 & 0 & 0 & 0 & 0 & 1 & 0
      \end{smallmatrix}\right).
    \end{equation*}
    This determines a quantum differential operator $Q(q, \theta) = \sum_{i=0}^6q^i Q_i(\theta)$ which annihilates $\tilde{J}^{\mathcal{E}'}_{X}$ $($see Appendix $ \ref{qde} $ for the explicit description of $Q_i(\theta))$.
  \end{Proposition}
  \begin{proof}
    From the definition of the quantum connection, we have  
    \begin{equation*}
      \nabla^z_{\frac{d}{dt}}T_i = z^{-1} T_1 \ast_{\mathcal{E}'} T_i = z^{-1} \sum_{j=0}^7 \sum_{d \geq 0} \langle T_1, T_i, T_j \rangle_{0,3,d}^{\mathcal{E}'} q^d T^j. 
    \end{equation*}
    Using the three point $\mathcal{E}'$-twisted Gromov--Witten invariants $\langle T_1, T_i, T_j \rangle_{0,3,d}^{\mathcal{E}'}$ in Lemma \ref{allGW}, it is straightforward to obtain the matrix $M$ representing $T_1 \ast_{\mathcal{E}'}$ (see (\ref{qdsystem})). If we have a connection matrix of $\nabla^z_{\frac{d}{dt}}$, the claimed differential operator $Q(q, \theta)$ follows from (\ref{deriv}).
    In Appendix \ref{qde}, we present the explicit form of $Q(q, \theta)$ by removing the common factor of $f_i(q)$ in (\ref{deriv}). 
  \end{proof}

 \begin{Proposition}\label{derivePF}
   The $I$-function $\tilde{I}_{\mathcal{E}'}^{\mathcal{O}(1)^{\oplus 2}}$ of a Calabi--Yau $3$-fold $Y$ of type No.\ $\sqmix$ satisfies the following differential equation:
   \begin{multline*}
     \{121\theta^4-11q(434\theta^4+820\theta^3+685\theta^2+275\theta+44)\\
     +q^2(7841\theta^4+10916\theta^3+3133\theta^2-2486\theta-1320)\\
     -4q^3(1488\theta^4+4092\theta^3+6761\theta^2+5511\theta+1694)\\
     -32q^4(136\theta^4+1460\theta^3+4556\theta^2+5245\theta+2017)\\
     +256q^5(\theta+1)^2(40\theta^2+212\theta+237)-4096q^6(\theta+1)^2(\theta+2)^2\} \tilde{I}_{\mathcal{E}'}^{\mathcal{O}(1)^{\oplus 2}} = 0.
   \end{multline*}
 \end{Proposition}

 \begin{proof}
   We define the differential operator 
\begin{equation*}
P(q, \theta) := \sum_{d=0}^6 q^d Q_d(\theta) \prod_{m=1}^d (\theta + m)^2. 
\end{equation*}
It follows from Lemma \ref{BvS} that $P(q, \theta)$ annihilates $\tilde{I}_{\mathcal{E}'}^{\mathcal{O}(1)^{\oplus 2}}$. 
Although $P(q, \theta)$ is of order $12$,
we find the following factorization in $\mathbb{Q}[q][\frac{d}{dq}]$: 
\begin{equation}\label{Weylfac}
P(q,\theta) = \theta^2 (\theta - 1)^2 \frac{1}{r(q)} R_4(q, \theta) S_4(q, \theta), 
\end{equation}
where $S_4(q, \theta)$ is the claimed Picard--Fuchs operator.
We refer the explicit form of $r(q)$ and $R_4(q, \theta)$ to Appendix \ref{factori}. 
Since $\tilde{J}_X^{\mathcal{E}' \oplus \mathcal{O}(1)^{\oplus 2}}$ corresponds to the flat section of a local system of rank 4,
$\tilde{I}_{\mathcal{E}'}^{\mathcal{O}(1)^{\oplus 2}}$ satisfies the fourth order differential equation which is given by $S_4(q, \theta)$. 
 \end{proof}
 
\section{No.\ $\dnc$ via the determinantal nets of conics}\label{dnc}
In this section, we consider
  the small $I$-function of a Calabi--Yau $3$-fold of No.\ $\dnc$ which is a zero locus of a general section of $\mathcal{E} = \Sym^2 \mathcal{S}^{\ast} \oplus \wedge^5 \mathcal{Q}$ on $G(2,8)$.
Neither method in Section \ref{abnonab} nor Section \ref{local} cannot be applied to this vector bundle $\mathcal{E}$.
To circumvent the situation, we note that the zero locus of a general section of $\wedge^5 \mathcal{Q}$ on $G(2,8)$ coincides with determinantal nets of conics $N$ which can be described by a geometric invariant theory quotient.
Moreover the restriction of $\Sym^2 \mathcal{S}^{\ast}$ to $N$ is a vector bundle on $N$ which is induced by a representation of the reductive algebraic group of the geometric invariant theory quotient.
If the abelian/nonabelian correspondence is true for $N$, then we can apply it for $\Sym^2 \mathcal{S}^{\ast}|_N$ and obtain 
the $\Sym^2 \mathcal{S}^{\ast}|_N$-twisted $J$-function on $N$ from the $\Sym^2 \mathcal{S}^{\ast}|_N$-twisted $I$-function. 

In this section, assuming the abelian/nonabelian correspondence for $N$, we determine the Picard--Fuchs operator corresponding to No.\ $\dnc$. 

\subsection{Construction of the determinantal nets of conics}
First we summarize the construction of the determinantal nets of conics in terms of a geometric invariant theory quotient following \cite{EPS}, \cite{Tjo1}. 

Let $F$ be a two-dimensional $\mathbb{C}$-vector space and $E$ be a three-dimensional $\mathbb{C}$-vector space.
Let $V = H^0(\mathbb{P}^2, \mathcal{O}(1))$ be the $\mathbb{C}$-vector space of linear polynomials on $\mathbb{P}^2$.
Choosing a basis of $F$ and $E$, we identify $\Hom(F,E \otimes V)$ with an affine space of $3 \times 2$ matrices having entries in $V$.
The group $GL(3, \mathbb{C}) \times GL(2, \mathbb{C})$ acts on $\Hom(F,E \otimes V)$ by 
\begin{equation*}
  (g,h) \cdot M = g M h^{-1}
\end{equation*}
for $(g,h) \in GL(3, \mathbb{C}) \times GL(2, \mathbb{C})$ and $M \in \Hom(F,E \otimes V)$.
Let $\mathbb{C}^{\ast} \cong \{ (\lambda I_3, \lambda I_2) \mid \lambda \in \mathbb{C}^{\ast}\}$ be a subgroup of the center of $GL(3, \mathbb{C}) \times GL(2, \mathbb{C})$.
Since $\mathbb{C}^{\ast}$ acts on $\Hom(F,E \otimes V)$ trivially, the above action induces an action of $G := GL(3, \mathbb{C}) \times GL(2, \mathbb{C})/ \mathbb{C}^{\ast}$ on $\Hom(F,E \otimes V)$.
We define a character $\chi:G \rightarrow \mathbb{C}^{\ast}$ by 
\begin{equation*}
  [(g,h)] \mapsto (\det g)^2 (\deg h)^{-3},
\end{equation*}
where $[(g,h)]$ is the equivalence class of $(g,h)$.
The geometric invariant theory quotient
\begin{equation*}
  N:=\Hom(F, E \otimes V)/\!\!/_{\chi}G
\end{equation*}
is known to be a six-dimensional variety. 
Furthermore the following properties are known: 
\begin{Proposition}[\cite{EPS}]\label{EPS}
  Under the above notation,
  it holds that
  \begin{enumerate}
  \item[$(i)$] $\Hom(F, E \otimes V)^{\text{ss}}(G) = \Hom(F, E \otimes V)^{\text{s}}(G)$,
  \item[$(ii)$] The action of $G$ on $\Hom (F, E \otimes V)^{\text{s}} (G)$ is fixed-point free.
  \end{enumerate}
  Moreover, $N$ is a smooth projective variety and the codimension of $ \Hom(F, E \otimes V) \setminus \Hom(F, E \otimes V)^{\text{ss}}(G)$ in $\Hom(F, E \otimes V)$ is larger than or equal to two $($cf.\ Assumption $\ref{assabnonab})$.
\end{Proposition}

Consider the standard representation $E$ of $GL(3, \mathbb{C})$, the standard representation $F$ of $GL(2, \mathbb{C})$, and the following induced representation of $G$: 
\begin{align}
  &E \otimes (\det E)^{-1} \otimes \det F,\label{etype} \\
  &F \otimes (\det E)^{-1} \otimes \det F.\label{ftype} 
\end{align}
These representations define corresponding homogeneous vector bundles $\mathcal{E}_N$ and $\mathcal{F}_N$ on $N$ respectively.
  Note that $\det \mathcal{E}_N^{\ast} = \det \mathcal{F}_N^{\ast}$ holds and the Picard group of $N$ is generated by this class. 

The variety $N$ parametrizes two-dimensional linear subsystems of conics of determinantal type. Hence there is a natural embedding
\begin{equation*}
  j:N \hookrightarrow G(3,6).
\end{equation*}
As in \cite{Tjo1}, we know $\mathcal{E}_N = j^{\ast} \mathcal{S}_{G(3,6)}$, where $\mathcal{S}_{G(3,6)}$ is the tautological subbundle on $G(3,6)$. For $\mathcal{F}_N$, we have obtained the following proposition:
\begin{Proposition}[{\cite[Proposition 6.1]{IIM}}]\label{iimres}
  There is an embedding $i:N \hookrightarrow G(2,8)$ such that
  \begin{equation}\label{ipull}
    i^{\ast} \mathcal{S}_{G(2,8)} = \mathcal{F}_N,
  \end{equation}
  where $\mathcal{S}_{G(2,8)}$ is the tautological subbundle on $G(2,8)$ and the image $i(N)$ is given by the zero locus of a general section of $\wedge^5 \mathcal{Q}$ on $G(2,8)$. 
\end{Proposition}
Using Proposition \ref{iimres}, we have
\[
\begin{CD}
  \Sym^2 \mathcal{F}^{\ast}_N @. \Sym^2 \mathcal{S}^{\ast} \oplus \wedge^5 \mathcal{Q} \\
  @VVV @VVV \\
  N @>>>  G(2,8)
\end{CD}
\]
and use this diagram for computation of the small $J$-function of No.\ $\dnc$. 
We can define a $\Sym^2 \mathcal{F}^{\ast}_N$-twisted $I$-function by Definition \ref{Ifct}. 
Although Theorem \ref{ItoJ} are proved for partial flag manifolds of type $A$,
we assume that the same statement holds for determinantal nets of conics and calculate the $\Sym^2 \mathcal{F}^{\ast}_N$-twisted small $J$-function.

\begin{Remark}
  In \cite{Tjo1}, Tj\o tta studies Calabi--Yau 3-folds associated with the following vector bundles on $N$:
  \begin{align*}
    &(\text{a}) \; \mathcal{O}(1)^{\oplus 3}, & &(\text{b}) \; \mathcal{F}^{\ast}_N \oplus \mathcal{O}(2), & &(\text{c}) \; \Sym^2 \mathcal{F}^{\ast}_N. &
  \end{align*}
  The corresponding $I$-function for the case (a) has been determined using quantum Lefschetz theorem.
  Due to the relation (\ref{ipull}) and similar relations, our Calabi--Yau $3$-folds No.\ 17, 15 and $\dnc$ correspond to the cases (a), (b) and (c) respectively.
  It should be noted that quantum Lefschetz theorem does not apply to the case (c) (in this given form), since the construction of Calabi--Yau $3$-fold does not factor through a Fano manifold. 
\end{Remark}

\subsection{Conjectural abelian/nonabelian correspondence for determinantal nets of conics}
\subsubsection{Description of the abelian quotient}
Let $T = (\mathbb{C}^{\ast})^3 \times (\mathbb{C}^{\ast})^2/\mathbb{C}^{\ast}$ be a maximal torus of $G$ and $W = N(T)/T \cong \mathfrak{S}_3 \times \mathfrak{S}_2$ be the Weyl group. 
Then $W$ acts on $\Hom(F, E \otimes V)$ by permutations of rows and columns.
The following lemma follows from the proof of Proposition \ref{EPS}:
\begin{Lemma}
  It holds that $\Hom(F, E \otimes V)^{\text{ss}}(T) = \Hom(F, E \otimes V)^{\text{s}}(T)$ for the actions of $T$.
  Moreover $\Hom(F, E \otimes V)^{\text{us}} := \Hom(F, E \otimes V) \setminus \Hom(F, E \otimes V)^{\text{ss}}(T)$ is given by  
  \begin{equation*}
    \Hom(F, E \otimes V)^{\text{us}} = \bigcup_{\sigma \in W} \left( \sigma \cdot 
    \mathbb{C}^{12}_{\left( \begin{smallmatrix}
        0 & 0 \\
        \ast & \ast \\
        \ast & \ast
        \end{smallmatrix} \right)}
    \cup
    \sigma \cdot
    \mathbb{C}^{12}_{\left( \begin{smallmatrix}
        0 & \ast \\
        0 & \ast \\
        \ast & \ast
        \end{smallmatrix} \right)}
    \right),
  \end{equation*}
  where
  \begin{equation*}
    \mathbb{C}^{12}_{\left( \begin{smallmatrix}
        0 & 0 \\
        \ast & \ast \\
        \ast & \ast
      \end{smallmatrix} \right)} = \left\{
    \left(
    \begin{smallmatrix}
      0 & 0 \\
      \ell_1 & \ell_2 \\
      \ell_3 & \ell_4
    \end{smallmatrix}
    \right) \  \Big|  \ \ell_1, \dots, \ell_4: \text{linear polynomials on } \mathbb{P}^2
    \right\}
  \end{equation*}
  is the affine subspace in $\Hom(F, E \otimes V)$, and similarly for
  $\mathbb{C}^{12}_{\left( \begin{smallmatrix}
        0 & \ast \\
        0 & \ast \\
        \ast & \ast
    \end{smallmatrix} \right)}$. 
\end{Lemma}

We put $\mathbb{P}_{\Delta} := \Hom(F, E \otimes V)/\!\!/_{\chi|_T}T$ for the abelian quotient by $T$.
We can check that the action of $T$ on $\Hom(F, E \otimes V)^{\text{s}}(T)$ is fixed-point free
and the toric variety $\mathbb{P}_{\Delta}$ is smooth and projective.
Hence the triple $(\Hom(F,E \otimes V), G,\chi)$ satisfies Assumption \ref{assabnonab}.

\subsubsection{Presentation of the cohomology ring of $N$ and $\mathbb{P}_{\Delta}$}
The structure of the cohomology ring of $N$ is described in \cite{ES}, \cite{Tjo1}.
It is generated by $p_i = c_i(\mathcal{E}_{N}^{\ast})$ $(i=1,2,3)$, $q_j = c_j(\mathcal{F}_{N}^{\ast})$ $(j=1,2)$ and the relations of these generators are completely known.
In particular, we can take a basis of $H^{\ast}(N, \mathbb{Q})$ as 
\begin{equation}\label{dncbasis}
  1, q_1, q_1^2, q_2, p_2, q_1^3, q_1 q_2, q_1 p_2, q_1^4, q_2^2, q_2 p_2, q_1 q_2^2, q_2^3
\end{equation}
from \cite[Theorem 6.9]{ES}. 

Since $\mathbb{P}_{\Delta}$ is a smooth toric variety, we can describe the cohomology ring of $\mathbb{P}_{\Delta}$ combinatorially (cf.\ \cite{Ful}).
We denote the Cox coordinates of $\mathbb{P}_{\Delta}$ by $z_{ij}^{\alpha}$ with $\ell_{ij} = \sum_{\alpha=0}^2 z_{ij}^{\alpha} x_{\alpha}$ and 
$\begin{pmatrix}
  \ell_{11} & \ell_{12} \\
  \ell_{21} & \ell_{22} \\
  \ell_{31} & \ell_{32}
\end{pmatrix}$. 
Let $H_{ij}$ be the cohomology class of the toric divisor of $\mathbb{P}_{\Delta}$ defined by $z_{ij}^{\alpha} = 0$ for some $\alpha$. 
We have the following lemma from \cite[Section 5.2]{Ful}:

\begin{Lemma}
  The cohomology ring $H^{\ast}(\mathbb{P}_{\Delta}, \mathbb{Z})$ is generated by $H_{ij}$ $(1 \leq i \leq  3, 1 \leq j \leq 2)$ and presented by
  \begin{equation*}
    H^{\ast}(\mathbb{P}_{\Delta}, \mathbb{Z}) = \mathbb{Z}[H_{ij} \mid 1 \leq i \leq 3, 1 \leq j \leq 2]/(I_{\text{Lin}} + I_{\text{SR}}),
  \end{equation*}
  where
  \begin{align*}
    I_{\text{Lin}} &= ( -H_{11} + H_{12} + H_{31} - H_{32}, -H_{21} + H_{22} + H_{31} - H_{32} ), \\
    I_{\text{SR}} &= ( H_{11}^3 H_{12}^3, H_{21}^3 H_{22}^3, H_{31}^3 H_{32}^3, H_{11}^3 H_{21}^3, H_{11}^3 H_{31}^3, H_{21}^3 H_{31}^3, H_{12}^3 H_{22}^3, H_{12}^3 H_{32}^3, H_{22}^3 H_{32}^3 )
  \end{align*}
  are ideals in $\mathbb{Z}[H_{ij} \mid 1 \leq i \leq 3, 1 \leq j \leq 2]$.
\end{Lemma}

Note that the Weyl group $W = \mathfrak{S}_3 \times \mathfrak{S}_2$ acts on $H^{\ast}(\mathbb{P}_{\Delta}, \mathbb{C})$ by permutations of $\{H_{ij}\}_{1 \leq i \leq 3, 1 \leq j \leq 2}$. 

\subsubsection{Decomposition of vector bundles}
We consider the representations of $T$ restricting the representations of $G$ given (\ref{etype}) and (\ref{ftype}).
Let $\mathcal{E}_{N, T}$, $\mathcal{F}_{N, T}$ be the corresponding vector bundles on $\mathbb{P}_{\Delta}$. We can write 
\begin{align*}
  \mathcal{E}_{N, T}^{\ast} &= \mathcal{O}(H_{11} + H_{22}) \oplus \mathcal{O}(H_{21} + H_{32}) \oplus \mathcal{O}(H_{31} + H_{12}), \\
  \mathcal{F}_{N, T}^{\ast} &= \mathcal{O}(H_{11} + H_{22} + H_{31}) \oplus \mathcal{O}(H_{12} + H_{21} + H_{32}). 
\end{align*}

\subsubsection{Fundamental Weyl anti-invariant class}
Let $\Phi$ be a root system corresponding to $(G,T)$.
We fix a decomposition $\Phi = \Phi_+ \amalg \Phi_{-}$ into positive roots and negative roots.
Then the fundamental Weyl anti-invariant class $\omega$ is written as
\begin{align*}
  \omega &= \prod_{\alpha \in \Phi_+} c_1(\mathcal{L}_{\alpha}) \\
  &= (H_{11}-H_{21})(H_{11}-H_{31})(H_{21}-H_{31})(H_{12}-H_{11}). 
\end{align*}

\subsubsection{Weyl invariant lift of cohomology classes of $N$}
We define a Weyl invariant lift of the basis in (\ref{dncbasis}) by 
\begin{equation*}
  \widetilde{p_i^a q_j^b} := \tilde{p_i}^a \tilde{q_j}^b \in H^{\ast}(\mathbb{P}_{\Delta}, \mathbb{C})^W,
\end{equation*}
where $\tilde{p_i} = c_i(\mathcal{E}_{N, T}^{\ast})$ ($i = 1,2,3$) and $\tilde{q_j} = c_j(\mathcal{F}_{N, T}^{\ast})$ $(j = 1,2)$.
Then, for example, the lift of $q_1 = c_1(\mathcal{F}_{N}^{\ast})$ is given by 
\begin{equation*}
  c_1(\mathcal{F}_{N, T}^{\ast}) = H_{11}+H_{12}+H_{21}+H_{22}+H_{31}+H_{32},
\end{equation*}
and this gives an isomorphism $\Pic N \cong (\Pic \mathbb{P}_{\Delta})^W$.

\subsubsection{Novikov rings}
We describe $\NE(\mathbb{P}_{\Delta})$.
Since $\mathbb{P}_{\Delta}$ is a toric variety, it is known that $\NE(\mathbb{P}_{\Delta})$ is generated by the class of the torus invariant curves (cf.\ \cite{CLS}).
We identify the numerical class of a curve with the intersection numbers with divisors $H_{ij}$, i.e.\ $C = (C. H_{ij}) \in \mathbb{Z}^6$. 
Then we have the following description of $\NE(\mathbb{P}_{\Delta})$: 
\begin{eqnarray*}
  \NE(\mathbb{P}_{\Delta}) \cong \{d = (d_{ij}) \in \mathbb{Z}^6 \mid -d_{11}+d_{12}+d_{21}-d_{22} = 0, -d_{11}+d_{12}+d_{31}-d_{32} = 0, \qquad \\
  d_{11}+d_{22} \geq 0, d_{21} + d_{32} \geq 0, d_{31} + d_{12} \geq 0, \qquad\\
  d_{11} + d_{22} + d_{31} \geq 0, d_{12} + d_{21} + d_{32} \geq 0\}. 
\end{eqnarray*}
The Novikov ring of $\mathbb{P}_{\Delta}$ is a completion of the semigroup ring $\mathbb{C}[\NE(\mathbb{P}_{\Delta})]$ by the valuation $v(Q^d) = \sum_{i,j} d_{ij}$. 

Since $N$ has Picard number one, the Novikov ring of $N$ is isomorphic to the formal power series ring $\mathbb{C}[\![{\tt{Q}}]\!]$ generated by one variable ${\tt{Q}}$.
It is easy to see that the homomorphism of Novikov rings $p:\Lambda_{\mathbb{P}_{\Delta}} \rightarrow \Lambda_N$ is given by
\begin{equation*}
p(Q^d)=(-1)^{\sum_{i,j} d_{ij}} {\tt{Q}}^{\sum_{i,j} d_{ij}}
\end{equation*}
for $Q^d \in \Lambda_{\mathbb{P}_{\Delta}}$. 

\subsubsection{$\Sym^2 \mathcal{F}_{N, T}^{\ast}$-twisted small $I$-function}
Note that $\Sym^2 \mathcal{F}_{N, T}^{\ast}$ is $\bigoplus_{i=1}^3 \mathcal{O}(D_i)$, where 
\begin{align*}
  D_1 &= 2H_{11} + 2H_{22} + 2H_{31},\\
  D_2 &= H_{11} + H_{12} + H_{21} + H_{22} + H_{31} + H_{32},\\
  D_3 &= 2H_{12} + 2H_{21} + 2H_{32}. 
\end{align*}
For $\tilde{t} \in H^2(\mathbb{P}_{\Delta}, \mathbb{C})$, 
a $\Sym^2 \mathcal{F}_{N, T}^{\ast}$-twisted small $I$-function is given by 
\begin{multline*}
  I_{\mathbb{P}_{\Delta}}^{\Sym^2 \mathcal{F}^{\ast}_{N, T}}(\tilde{t}, z) =
  z e^{\tilde{t}/z} \sum_{d \in \NE(\mathbb{P}_{\Delta})} \frac{\prod_{1 \leq i \leq 3, 1 \leq j \leq 2} \prod_{m=-\infty}^{m=0}(H_{ij} + m z)^3}{\prod_{1 \leq i \leq 3, 1 \leq j \leq 2} \prod_{m=-\infty}^{m=d_{ij}}(H_{ij} + m z)^3}\\
  \prod_{m=1}^{D_1.d} (D_1 + mz) \prod_{m=1}^{D_2.d}(D_2 + mz) \prod_{m=1}^{D_3.d}(D_3 + mz) Q^d e^{\tilde{t}.d}
\end{multline*}
which is a function on $H^2(\mathbb{P}_{\Delta}, \mathbb{C})$ and $1/z$ taking its values in $H^{\ast}(\mathbb{P}_{\Delta}, \mathbb{C})$. 
We write $I_{\mathbb{P}_{\Delta}}^{\Sym^2 \mathcal{F}^{\ast}_{N, T}} = z e^{\tilde{t}/z} \sum_{d \in \NE(\mathbb{P}_{\Delta})} I_{d} Q^d$ for simplicity. 

\subsubsection{$\Sym^2 \mathcal{F}_{N}^{\ast}$-twisted  small $I$-function}
Using the above data, we calculate the conjectural formula of $\Sym^2 \mathcal{F}_N^{\ast}$-twisted small $I$-function by (\ref{formIfct}). 
Then we have 
\begin{align}\label{finresult}
  I_{N}^{\Sym^2 \mathcal{F}_N^{\ast}}(t,z) &= \frac{1}{\omega}\Biggl(\Bigl(\prod_{\alpha \in \Phi_+} z \partial_{\alpha}\Bigr) I_{\mathbb{P}_{\Delta}}^{\Sym^2 \mathcal{F}^{\ast}_{N,T}}\Biggr)\Biggr|_{\tilde{t} = \sum_{i,j}tH_{ij}, Q^d = (-1)^{\sum_{i,j}d_{ij}}{\tt{Q}}^{\sum_{i,j}d_{ij}}}\nonumber\\
  &= \frac{1}{\omega} z e^{\sum_{i,j}t H_{ij}/z} \sum_{d \in \NE(\mathbb{P}_{\Delta})} (-1)^{\sum_{i,j}d_{ij}} (H_{11}-H_{21} + z(d_{11}-d_{21}))(H_{11}-H_{31} + z(d_{11}-d_{31}))\nonumber\\
  &\qquad\qquad (H_{21}-H_{31} + z(d_{21}-d_{31}))(H_{12}-H_{11} + z(d_{12}-d_{11}))I_{d} {\tt{Q}}^{\sum_{i,j}d_{ij}} ,
\end{align}
where $t$ is the coordinates on $H^2(N, \mathbb{C})$ corresponding to the generator $H := c_1(\mathcal{F}_N^{\ast})$ of $\Pic N$.

\subsubsection{Quantum differential equation}
By the identification of Weyl anti-invariant classes with Weyl invariant lifts given by (\ref{liftnote}), we consider the $\Sym^2 \mathcal{F}_N^{\ast}$-twisted small $I$-function (\ref{finresult}) which takes values in $H^{\ast}(N, \mathbb{C})$.
We expand $\tilde{I}_{N}^{\Sym^2 \mathcal{F}_N^{\ast}}$ with respect to $z$.
Then we have 
\begin{equation*}
  \tilde{I}_{N}^{\Sym^2 \mathcal{F}_N^{\ast}}(t,z) = I_0(t) + I_1(t) H + I_2(t) H^2 z^{-1} + I_3(t) H^3 z^{-2}, 
\end{equation*}
where we regard $H$ as the class in $H^{\ast}(N, \mathbb{C}) / \ann(e(\Sym^2 \mathcal{F}_N^{\ast}))$. 
We evaluate ${\tt{Q}} = 1$ and $q = e^t$.
The functions $I_0(t)$ and $I_1(t)$ are given by 
\begin{align*}
  I_0(t) &= 1 + 6 q + 66 q^2 + 1092 q^3 +\cdots ,\\
  I_1(t) &= I_0(t) t + 10 q + 167 q^2 + \frac{26746}{3} q^3 + \cdots . 
\end{align*}
We have the following result by searching a differential operator which annihilates $I_0$ similar to Example \ref{exana}: 
\begin{Proposition}
The function $\tilde{I}_{N}^{\mathrm{Sym}^2 \mathcal{F}_{N}^{\ast}}(t)$ satisfies the differential equation 
\begin{multline*}
  \{\theta^4-2q \left(2 \theta^2+2 \theta+1\right)
   \left(11 \theta^2+11 \theta+3\right)\\
  +4q^2 (\theta+1)^2 \left(76 \theta^2+152 \theta+111\right) -144 q^3 (\theta+1) (\theta+2) (2 \theta+3)^2\} \tilde{I}_{N}^{\mathrm{Sym}^2 \mathcal{F}_N^{\ast}} = 0. 
\end{multline*}
\end{Proposition}
If Theorem \ref{ItoJ} holds for determinantal nets of conics, we can obtain the $J$-function of Calabi--Yau $3$-fold of No.\ $\dnc$ from $\tilde{I}_{N}^{\mathrm{Sym}^2 \mathcal{F}_{N}^{\ast}}$.

\appendix
\section{List of Picard--Fuchs operators}\label{pflist}

We list differential operators which annihilate 
  the $I$-functions of Calabi--Yau $3$-folds corresponding to the pairs $(X, \mathcal{E})$ in Table \ref{cylist}. 
In \cite{BCFKvS}, we can find a conjectural mirror family for complete intersection Calabi--Yau $3$-folds in Grassmannian by using the conifold transition as well as their Picard--Fuchs operators.
To avoid overlap with their list, we omit the Picard--Fuchs operators for complete intersection Calabi--Yau $3$-folds in Grassmannian which correspond to No.\ 1, 2, 3, 6, 12, 19.
Except for this omission, for completeness, we include previously known examples in other literatures.

In the list below, we set $\theta = q \frac{d}{dq}$ and $Y$ to be the zero locus of a general section of $\mathcal{E}$ on $X$.
\bigskip

\paragraph{No.\ 4 : $X = G(2,5), \mathcal{E} = \mathcal{S}^{\ast}(1) \oplus \mathcal{O}(2)$. }
\begin{align*}
  \theta^4 -2 q (2 \theta+1)^2 \left(17 \theta^2 + 17 \theta + 5\right)  + 4 q^2 (\theta + 1)^2 (2 \theta + 1) (2 \theta + 3)
\end{align*}
It is known that the zero locus $Y$ is a complete intersection of $OG(5,10)$ (see the description of $V_{12}$ in \cite[Section 13]{CCGK}). 
\bigskip

\paragraph{No.\ 5 : $X = G(2,5), \mathcal{E} = \wedge^2 \mathcal{Q}(1)$. }
\begin{multline*}
  \theta^4 - q \left(124 \theta^4+242 \theta^3+187 \theta^2+66
  \theta+9\right)\\
  +q^2 \left(123 \theta^4-246\theta^3-787 \theta^2-554 \theta-124\right)\\
  +q^3\left(123 \theta^4+738 \theta^3+689 \theta^2+210 \theta+12\right)\\
  -q^4 \left(124 \theta^4+254 \theta^3+205 \theta^2+78 \theta+12\right)
  +q^5 (\theta+1)^4
\end{multline*}
It is proved in \cite[Proposition 4.7]{IIM} that the zero locus $Y$ is deformation equivalent to the complete intersection of two $G(2,5)$ in $\mathbb{P}^9$ which is studied in \cite{GP}, \cite{Kan}.
The Picard--Fuchs equation of the Calabi--Yau 3-fold is derived in \cite{Kap2} via conifold transition.
\bigskip

\paragraph{No.\ 7 : $X = G(2,6), \mathcal{E} = \mathcal{S}^{\ast}(1) \oplus
    \mathcal{O}(1)^{\oplus 3}$. }
\begin{multline*}
  121 \theta^4 -77 q \left(130 \theta^4+266 \theta^3+210 \theta^2+77 \theta+11\right)\\
  - q^2 \left(32126 \theta^4+89990 \theta^3+103725 \theta^2+ 55253 \theta+ 11198\right)\\
  -q^3 \left(28723 \theta^4+ 74184 \theta^3+ 63474 \theta^2+ 20625 \theta+ 1716\right)\\
  -7 q^4 \left(1135 \theta^4+2336 \theta^3+1881 \theta^2+713 \theta+110\right) - 49 q^5 (\theta+1)^4
\end{multline*}
The zero locus $Y$ is isomorphic to a linear section of a minuscule Schubert variety in $\mathbb{OP}^2$ (see \cite[Section 4.1]{IIM}),
which is studied in \cite{Miu}, \cite{Gal}, \cite{GKM}. 
\bigskip

\paragraph{No.\ 10 : $X = G(2,6), \mathcal{E} = \mathcal{Q}(1) \oplus \mathcal{O}(1)$. }

In \cite{Man}, Manivel shows that $Y$ is isomorphic to a general  linear section of $G(2,7)$ of codimension $7$.
Since the $I$-function which is given by the abelian/nonabelian correspondence is same as No.\ 12, we omit the Picard--Fuchs operator of it. 
\bigskip

\paragraph{No.\ 13 : $X = G(2,7), \mathcal{E} = \Sym^2 \mathcal{S}^{\ast} \oplus \mathcal{O}(1)^{\oplus 4}$. }
\begin{multline*}
  49 \theta^4-14 q \left(134 \theta^4+286 \theta^3+234 \theta^2+91 \theta+14\right)\\
  -4 q^2 \left(3183 \theta^4+10266 \theta^3+13501 \theta^2+8225
  \theta+1918\right)\\
  -8 q^3 \left(2588 \theta^4+8400 \theta^3+10256 \theta^2+5649
  \theta+1190\right)\\
  -48 q^4 \left(256 \theta^4+848 \theta^3+1141 \theta^2+717
   \theta+174\right)-2304 q^5 (\theta+1)^4
\end{multline*}
The zero locus $Y$ is a complete intersection of the orthogonal Grassmannian $OG(2,7)$.
\bigskip

\paragraph{No.\ 15 : $X = G(2,7), \mathcal{E} = \wedge^4 \mathcal{Q} \oplus \mathcal{O}(1) \oplus \mathcal{O}(2)$. }
\begin{equation*}
  \theta^4-6 q (2 \theta+1)^2 \left(3 \theta^2+3 \theta+1\right)-108 q^2 (\theta+1)^2 (2 \theta+1) (2 \theta+3)
\end{equation*}
The zero locus of a general section of $\wedge^4 \mathcal{Q}$ is isomorphic to the homogeneous space $G_2/P_{\text{long}}$ (see the description of $V_{18}$ in \cite[Section 16]{CCGK}).
Hence $Y$ is a complete intersection of the homogeneous space.
The Picard--Fuchs operator is already known by \cite{Kap2} via conifold transition. 
\bigskip

\paragraph{No.\ \linseq : $X = G(2,7), \mathcal{E} = \mathcal{S}^{\ast}(1) \oplus \wedge^4 \mathcal{Q}$. }
It is proved in \cite[Proposition 5.1]{IIM} that $Y$ is deformation equivalent to a linear section of $G(2,7)$. Since Gromov--Witten invariants are deformation invariant, the $J$-function is the same as that of No.\ 12.
\bigskip

\paragraph{No.\ 17 : $X = G(2,8), \mathcal{E} = \wedge^5 \mathcal{Q} \oplus \mathcal{O}(1)^{\oplus 3}$. }
\begin{multline*}
  361 \theta^4 - 19 q \left(700 \theta^4+1238 \theta^3+999 \theta^2+380 \theta+57\right)\\
  + q^2 \left(-64745 \theta^4-368006 \theta^3-609133 \theta^2-412756
  \theta-102258\right)\\
  + 27 q^3 \left(6397 \theta^4+12198 \theta^3-11923 \theta^2-27360 \theta-11286\right)\\
  + 729 q^4 \left(64 \theta^4+1154 \theta^3+2425 \theta^2+1848\theta+486\right) - 177147 q^5 (\theta+1)^4
\end{multline*}
The zero locus $Y$ is isomorphic to a linear section of the determinantal nets of conics which is studied by Tj\o tta. Tj\o tta has calculated the Picard--Fuchs operator  in \cite{Tjo1} by using the direct calculation of the quantum connection of $N$ and using the quantum Lefschetz theorem.
\bigskip

\paragraph{No.\ 18 : $X = G(2,8), \mathcal{E} = \Sym^2 \mathcal{S}^{\ast} \oplus \wedge^5 \mathcal{Q}$. }
\begin{multline*}
  \theta^4-2q \left(2 \theta^2+2 \theta+1\right)
   \left(11 \theta^2+11 \theta+3\right)\\
  +4q^2 (\theta+1)^2 \left(76 \theta^2+152 \theta+111\right) -144 q^3 (\theta+1) (\theta+2) (2 \theta+3)^2
\end{multline*}
The zero locus $Y$ has appeared first in \cite{Tjo1}, and the virtual number of lines and conics on $Y$ are computed in \cite{Tjo1}.
Our computations in Section \ref{dnc} are consistent to these numbers.
\bigskip

\paragraph{No.\ 20 : $X = G(3,6), \mathcal{E} = \wedge^2 \mathcal{S}^{\ast} \oplus \mathcal{O}(1)^{\oplus 2} \oplus \mathcal{O}(2)$. }
\begin{equation*}
\theta^4-8 q (2 \theta+1)^2 \left(3 \theta^2+3 \theta+1\right)+64 q^2 (\theta+1)^2 (2 \theta+1) (2 \theta+3)
\end{equation*}
The zero locus of a general section of $\wedge^2 \mathcal{S}^{\ast}$ is the Lagrangian Grassmannian $LG(3,6)$. 
Hence $Y$ is a complete intersection on the homogeneous space $LG(3,6)$. 
The above differential operator coincides with the predicted one by \cite{vEvS1}.
\bigskip

\paragraph{No.\ 21 : $X = G(3,6), \mathcal{E} = \mathcal{S}^{\ast}(1) \oplus \wedge^2 \mathcal{S}^{\ast}$. }
Similarly to No.\ \linseq, $Y$ is deformation equivalent to linear sections of $G(3,6)$ by \cite[Proposition 5.1]{IIM}. The Picard--Fuchs operator can be found in \cite{BCFKvS}.
\bigskip

\paragraph{No.\ 22 : $X = G(3,7), \mathcal{E} = \Sym^2 \mathcal{S}^{\ast} \oplus \mathcal{O}(1)^{\oplus 3}$. }
The zero locus $Z$ of a general section of $\Sym^2 \mathcal{S}^{\ast}$ is isomorphic to orthogonal Grassmannian $OG(3,7)$. Using the natural isomorphisms $OG(3,7) \cong OG(4,8) \cong OG(1,8)$,
$Z$ is isomorphic to a smooth quadric hypersurface $Q$ of $\mathbb{P}^7$. 
Under this isomorphism, the restriction of $\mathcal{O}_{G(3,7)}(1)$ on $Z$ coincides with $\mathcal{O}_{Q}(2)$. Hence $Y$ is isomorphic to a complete intersection of four quadric hypersurfaces of $\mathbb{P}^7$.
\bigskip

\paragraph{No.\ 23 : $X = G(3,7), \mathcal{E} = (\wedge^2 \mathcal{S}^{\ast})^{\oplus 2} \oplus \mathcal{O}(1)^{\oplus 3}$. }
\begin{multline*}
  3721 \theta^4-61 q \left(3029 \theta^4+5572 \theta^3+4677 \theta^2+1891
  \theta+305\right)\\
  +q^2 \left(1215215 \theta^4+3428132 \theta^3+4267228 \theta^2+2572675
  \theta+611586\right)\\
  -81 q^3 \left(39370 \theta^4+140178 \theta^3+206807 \theta^2+142191
  \theta+37332\right)\\
  +6561 q^4 \left(566 \theta^4+2230 \theta^3+3356 \theta^2+2241
  \theta+558\right)-1594323 q^5 (\theta+1)^4
\end{multline*}
\bigskip

\paragraph{No.\ 24 : $X = G(3,7), \mathcal{E} = (\wedge^3 \mathcal{Q})^{\oplus 2} \oplus \mathcal{O}(1)$. }
\begin{multline*}
  81 \theta^4-9 q \left(317 \theta^4+520 \theta^3+431 \theta^2+171 \theta+27\right)\\
  +q^2 \left(6589 \theta^4-7616 \theta^3-31688 \theta^2-28251
  \theta-8370\right)\\
  -q^3 \left(5521 \theta^4+21384 \theta^3+107223 \theta^2+138402
  \theta+55782\right)\\
  +q^4 \left(21987 \theta^4+130752 \theta^3+152168 \theta^2+9194
  \theta-39016\right)\\
  -19 q^5 (\theta+1) \left(293 \theta^3-7005 \theta^2-18780
  \theta-12535\right)\\
  -361 q^6 (\theta+1) (\theta+2) \left(137 \theta^2+357
  \theta+223\right)
  -6859 q^7 (\theta+1) (\theta+2)^2 (\theta+3)
\end{multline*}
\bigskip

\paragraph{No.\ 25 : $X = G(3,7), \mathcal{E} = \wedge^2 \mathcal{S}^{\ast} \oplus \wedge^3 \mathcal{Q} \oplus \mathcal{O}(1)^{\oplus 2}$. }
\begin{multline*}
  121\theta^4-11q(434\theta^4+820\theta^3+685\theta^2+275\theta+44)\\
  +q^2(7841\theta^4+10916\theta^3+3133\theta^2-2486\theta-1320)\\
  -4q^3(1488\theta^4+4092\theta^3+6761\theta^2+5511\theta+1694)\\
  -32q^4(136\theta^4+1460\theta^3+4556\theta^2+5245\theta+2017)\\
  +256q^5(\theta+1)^2(40\theta^2+212\theta+237)-4096q^6(\theta+1)^2(\theta+2)^2
\end{multline*}
\bigskip

%

\paragraph{No.\ 28 : $X = G(3,8), \mathcal{E} = (\wedge^2 \mathcal{S}^{\ast})^{\oplus 4}$. }
\begin{multline*}
  529 \theta^4-23 q \left(850
  \theta^4+1634 \theta^3+1461 \theta^2+644 \theta+115\right)\\
  +q^2 \left(140191
  \theta^4+504286 \theta^3+765193 \theta^2+554484 \theta+160080\right)\\
  -2 q^3 \left(225598
  \theta^4+1145682 \theta^3+2338529 \theta^2+2203584 \theta+805023\right)\\
  +q^4 \left(739023 \theta^4+4533564
  \theta^3+11008538 \theta^2+12265136 \theta+5219792\right)\\
  -q^5 \left(683438 \theta^4+4503734 \theta^3+11576207
  \theta^2+13658100 \theta+6109137\right)\\
  +q^6 \left(392337 \theta^4+2558982 \theta^3+6275869 \theta^2+6925308
  \theta+2877864\right)\\
  -4 q^7 \left(36774 \theta^4+235086 \theta^3+558499 \theta^2+579748
  \theta+221197\right)\\
  +16 q^8 \left(2081 \theta^4+12548 \theta^3+28814 \theta^2+29364
  \theta+11133\right)
  -256 q^9 (2 \theta+3)^4
\end{multline*}
\bigskip

\section{Some explicit formulas discussed in Section \ref{local}}

\subsection{Quantum differential equation of $\tilde{J}_{X}^{\mathcal{E}'}$}\label{qde}
We give an explicit form of the quantum differential equation of $\tilde{J}_{X}^{\mathcal{E}'}$ in Proposition \ref{deribeqde} as follows:
\begin{align*}
  Q_0 &= 242(\theta-1)^2\theta^6 ,\\
  Q_1 &= -11 \theta^2(1615\theta^6-4845\theta^5+4598\theta^4+1640\theta^3+1370\theta^2+550\theta+88), \\
  Q_2 &= 103968\theta^8-415872\theta^7+1108270\theta^6+662150\theta^5+2273\theta^4\\
   &\hspace{5cm}-230318\theta^3-24299\theta^2+61468\theta+19360,\\
  Q_3 &= -3950784\theta^6+4345395\theta^4+1309385\theta^3-2484702\theta^2 -1988272\theta-440880,\\
  Q_4 &= 4(233928\theta^4+467856\theta^3+664544\theta^2+587366\theta+156697),\\
  Q_5 &= -1216(1368\theta^2+2736\theta-479),\\
  Q_6 &= -6653952.
\end{align*}

\subsection{Factorization of differential operator}\label{factori}
The factors $r(q)$ and $   R_4(q, \theta) $ of the differential operator (\ref{Weylfac}) in the proof of Proposition \ref{derivePF}
are as follows:
\begin{align*}
  r(q) =& 32768 q^5-225280 q^4+619520 q^3-851840 q^2+585640 q-161051, \\
  R_4(q, \theta) =& 103968 q^2 (8 q-11)^3 \theta^4 -415872 q^2 (8 q-11)^2 (20 q-11) \theta^3\\
  & +19 q (8 q-11) \left(12957696 q^3-11215424 q^2+2923360 q-113135\right) \theta^2\\
  & -57 q \left(56033280 q^4-48468992 q^3+5056832 q^2-3620320 q+1244485\right) \theta\\
  & + 2 \left(958169088 q^5+892286208 q^4-67694176 q^3-91902888 q^2+27891105
   q-161051\right). 
\end{align*}

\end{document}